\documentclass[12pt,english,reqno]{amsart}
\usepackage{amsmath,amssymb,amsthm,amsfonts,graphicx,color}
\usepackage{amssymb,geometry}

\usepackage{color, graphics}

\usepackage{graphicx}
\usepackage[T1]{fontenc}
\usepackage[latin1]{inputenc}
\usepackage[english]{babel}
\usepackage{geometry}

\usepackage{cases}

\usepackage[colorlinks=true]{hyperref}

\usepackage[colorlinks=true]{hyperref}
\hypersetup{linkcolor=black,citecolor=black,filecolor=black,urlcolor=black} 

\newcommand{\blue}[1]{{\color{blue} #1}}

\newcommand{\x}{{\tt x}}
\newcommand{\y}{{\tt y}}

\newcommand{\R}{\mathbb{R}}

\newcommand{\diver}{\text{div}}

\newcommand{\mS}{\mathbb{S}}

 \newtheorem{lemma}{Lemma}[section]
\newtheorem{definition}{Definition}[section]
\newtheorem{theorem}{Theorem}[section]
\newtheorem{proposition}{Proposition}[section]

\newtheorem{remark}{Remark}[section]
\newcommand{\bremark}{\begin{remark} \em}
\newcommand{\eremark}{\end{remark} }

\numberwithin{equation}{section}

\title{fredholm properties of the jacobi Operator of minimal conical hypersurfaces}

\author{Oscar Agudelo}
\address{University of West Bohemia in Pilsen-NTIS, Univerzitn\'{i} 22, Czech Republic.}
\email {oiagudel@ntis.zcu.cz}

\author{Matteo Rizzi}
\address{Mathematisches Institut, Justus Liebig Universit\"{a}t, Arndtstrasse 2, 35392, Giessen, Germany.}
\email{mrizzi1988@gmail.com}
\thanks{O. Agudelo was supported by the Grant 22-18261S of the Grant Agency of the Czech Republic. M. Rizzi was partially supported by the Alexander von Humboldt foundation.}

\begin{document}

\maketitle

\begin{abstract}
In this paper we study non-degeneracy properties of $\Sigma$ via the Jacobi operator $J_\Sigma:=\Delta_\Sigma+|A_\Sigma|^2$ of a given minimal hypersurface $\Sigma$ asymptotic to a cone $C\subset \R^{N+1}$ of co-dimension one. Here $\Delta_{\Sigma}$ is the Laplace Beltrami operator of $\Sigma$ and $|A_{\Sigma}|$ is the norm of the second fundamental form of $\Sigma$. We also construct a right inverse of $J_{\Sigma}$, that is, we prove that  the Jacobi equation $J_\Sigma\phi=f$ is solvable in $\Sigma$, at least under some suitable non-degeneracy assumptions about $\Sigma$ and about the asymptotic behavior of $f$ at infinity. We also
discuss some
examples where our results can be applied.
\end{abstract}

\section{introduction}
Throughout this work we assume that $N\ge 3$ and we denote by $\mS^{N}$ the unit sphere in $\R^{N+1}$ centered at the origin. 

\vskip 3pt 
Let us start by introducing some notation. For any $m\in \mathbb{N}$, $O(m)$ denotes the group of orthogonal transformations of $\R^{m}$. Let $S$ be a subgroup of $O(N+1)$ 
A set $D\subset \R^{N+1}$ is $S-$invariant if for any $\rho\in S$, $\rho(D)\subset D$. If $D$ is an $S-$invariant set, a function $f:D\to \R$ is $S-$invariant if for any $\rho\in S$, $f\circ\rho=f$ in $D$.

\vskip 3pt
Next, let $\Gamma\subset \mS^{N}$ be a minimal submanifold of $\mS^N$ of codimension one and consider the minimal cone
$$
C:=\{r\theta\in \R^{N+1}\,:\, r>0 \quad \hbox{and} \quad \theta \in \Gamma\}.
$$

We will refer to $C$ as a trivial cone, whenever $\Gamma$ lies in a linear subspace of $\R^{N+1}$ of dimension $N$. In this case $C$ coincides with this subspace.

The vector function $\nu_C:C\setminus \{0\} \to \mathbb{S}^N$ will denote a fixed and continuous choice of the unit normal vector to $C\setminus \{0\}$. In view of the homogeneity of $C$ with respect to the variable $r$, $\nu_C$ is independent of $r$, i.e., $\nu_C(r\theta)=\nu_C(\theta)$ for any $r>0$ and $\theta \in \Gamma$.  

\vskip 3pt

Now, let $\Sigma$  be a smooth minimal hypersurface of codimension one. 
In what follows, we assume the following hypotheses on $C$ and $\Sigma$: 
\begin{itemize} 
\item[(H1)] $C$ is a nontrivial and $S-$invariant cone;
\vskip 3pt
\item[(H2)] $\Sigma$ is $S-$invariant;
\vskip 3pt
\item[(H3)] $\Sigma$ is asymptotic to the cone $C$, that is: there exist $R>0$ and a smooth function $w:(R,\infty)\times\Gamma\to \R$ such that $\lim \limits_{r\to \infty} \|w(r,\cdot)\|_{C^2(\Gamma)}=0$ and such that
\begin{equation}
\label{Sigma-normal-graph}
\Sigma_R:=\{r\theta+w(r,\theta)\nu_C(\theta):\,\theta\in\Gamma,\,r>R\}.
\end{equation}
and $K_R:=\Sigma\setminus\Sigma_R$ is compact.
\color{black}\item[(H4)] $S$ is maximal with respect to property (H1). In other words, if $\rho\in O(N+1)$ is such that $\rho(C)\subset C$, then $\rho\in S$.
\end{itemize}

The invariance assumption in (H1) is equivalent to saying that $\Gamma$ is $S-$invariant. This hypothesis is not restrictive, since $S$ may be the trivial subgroup. However, in the case that $C$ (or $\Sigma$) features certain symmetries, it is possible to find examples of such  nontrivial subgroups $S$. For instance, when $m,n\geq 2$ and $C=C_{m,n}$ is the well-known Lawson cone, then we may consider $S=O(m)\times O(n)$ (see Theorems 1.4, 1.5 below and the works \cite{agudelo2022doubling,agudelo2022k,agudelo2024jacobi}). We refer the reader to \cite{allard1981radial, simon2006isolated} for notions related to this invariance. Hypothesis (H3), is motivated by the Definition 2.1 and the developments in \cite{agudelo2024jacobi}. Up to modifying $S$ by means of the Zorn lemma, if necessary, it is not restrictive to assume that (H4) holds. This fact, together with hypothesis (H2), yields that $\Sigma$ inherits the symmetries of $C$. In the present work, we need a finer topology in order to extract more precise information of the geometric objects hereby studied. We believe that Definition 2.1 in \cite{agudelo2024jacobi} and the statement in (H3) are equivalent.

\vskip 3pt
Let $J_{\Sigma}$ denote the Jacobi operator of $\Sigma$, i.e.,
\begin{equation}
\label{Jacobi-Sigma-lin}J_\Sigma:=\Delta_{\Sigma} +|A_{\Sigma}|^2, 
\end{equation}
where $\Delta_\Sigma$ is the Laplace-Beltrami operator over $\Sigma$ and $|A_{\Sigma}|$ corresponds to the norm of the second fundamental form of $\Sigma$.

\vskip 3pt
Let $\mathcal{D}(\Sigma)=(C^\infty_c(\Sigma))'$ denote the space of distributions on $\Sigma$ (see, e.g., Chapter 5 in \cite{rudin1974functional}). We consider the {\it Jacobi equation} 
$J_{\Sigma}\phi = 0$ in $\mathcal{D}(\Sigma)$. In this sense, we allow for singular solutions in $\mathcal{D}(\Sigma)$ of the Jacobi equation (see e.g., section 5 in \cite{agudelo2024jacobi}). The solutions to the Jacobi equation will be referred to as {\it Jacobi fields} of $\Sigma$. 

\vskip 3pt
As it is known from  standard regularity theory and the smoothness of $\Sigma$, {\it Jacobi fields} $\phi \in L^{\infty}_{loc}(\Sigma)$, actually belong to $C^2(\Sigma)$ and solve the Jacobi equation in the classical sense.

\vskip 3pt
Following the definitions and conventions in \cite{agudelo2024jacobi} we recall that $G(\Sigma)$ denotes the space of geometric Jacobi fields in  $\Sigma$, that is, the space of locally bounded solutions of the Jacobi equation generated by the invariances of the minimal surface $\Sigma$. These invariances arise from translations, dilations and rotations of $\Sigma$ (see Introduction in \cite{agudelo2024jacobi}). Let $\{e_1,\ldots, e_{N+1}\}$ be the standard basis in $\R^{N+1}$ and consider the Jacobi fields of $\Sigma$ generated by the invariance of $\Sigma$ under translations, i.e.,
\begin{equation}\label{def:jac_field_transl}
\zeta_j(y):=e_j\cdot \nu_\Sigma(y) \quad \hbox{for } y\in \Sigma, \,1\leq j\leq N+1.
\end{equation}

Observe that for any $j=1,\ldots,N+1$, $\zeta_j\in L^{\infty}(\Sigma)$ and $\zeta_j$ is bounded away from zero at infinity.

Let also $\zeta_0$ be the Jacobi field of $\Sigma$ generated by the invariance under dilations of $\Sigma$, i.e,
\begin{equation}\label{def:jac_field_dilat}
\zeta_0(y):=y\cdotp\nu_\Sigma(y) \quad \hbox{for }y\in \Sigma.
\end{equation}

Let 
\begin{equation}\label{def:matrix_antisymm}
\mathcal{A}_{N+1}:=\{M\in\R^{N+1}\times\R^{N+1}:\,M=-M^t\}    
\end{equation}
be the space of antisymmetric $(N+1)\times(N+1)$-matrices. We observe that ${\rm dim}(\mathcal{A}_{N+1})=\frac{N(N+1)}{2}$. Then the space of the Jacobi fields arising from rotations is given by $\{\zeta_M\}_{M\in \mathcal{A}_{N+1}}$, where
{\begin{equation}
\label{jacobi-fields-rot}
\zeta_M(y):=My\cdotp\nu_\Sigma(y)\quad \hbox{for }y\in \Sigma.
\end{equation}}

Observe that for any matrix $M\in\mathcal{A}_{N+1}$ such that the corresponding rotation $\rho\in SO(N+1)$ fulfills $\rho(\Sigma)\neq \Sigma$, the Jacobi field $\zeta_M$ is unbounded (see Lemma \ref{lemma-dec-Jacobi}). 

\vskip 3pt

Since $G(\Sigma)={\rm span}\{\zeta_0,\zeta_1,\ldots,\zeta_{N+1}\}\oplus \{\zeta_M\}_{M\in\mathcal{A}_{N+1}}$, 
$$
{\rm dim}(G(\Sigma))\leq \frac{N(N+1)}{2}+N+2.
$$

For future reference, denote also by  $G(\Sigma,S)$ the subspace of $G(\Sigma)$ of $S$-invariant geometric Jacobi fields.

\vskip 3pt
In view of (H3) and the fact that the Jacobi field of $C$ associated to dilations is the trivial one, $\zeta_0$ decays to zero at infinity, as we will prove in Lemma \ref{lemma-dec-Jacobi}. We recall that the decay of $\zeta_0$ is polynomial if the underlying cone is strictly stable, however for more general cones we cannot expect this decay to be polynomial, even in the case of stable, but not strictly stable underlying cones (see Remark $3.2$ of \cite{agudelo2024jacobi}). 

Set
\begin{equation}\label{def:rate_decay_zeta0}
\bar{\nu}:=\inf\{\nu\in\R:\, \limsup_{|y|\to\infty}|y|^{-\nu}|\zeta_0(y)|=0\}.
\end{equation}

The infimum in \eqref{def:rate_decay_zeta0} is well defined in view of Lemma \ref{lemma-dec-Jacobi}. Observe also that $\bar{\nu}\leq 0$.

\medskip
Next, for any $\nu\in\R$ define
\begin{equation}\label{def:jacobi_fields_decay}
D_\nu(\Sigma):=\{\phi\in C^2(\Sigma):J_\Sigma \phi=0,\,\phi(1+|y|)^{-\nu}\in L^\infty(\Sigma)\}
\end{equation}
and
\begin{equation}\notag
D(\Sigma,S)_\nu:=\{\phi\in D_\nu(\Sigma):\,\phi\text{ is $S$-invariant}\}.
\end{equation}

\vskip 3pt
Our first result states that when the underlying cone $C$ is nontrivial, there are no nontrivial bounded geometric Jacobi fields of $\Sigma$, which decay at a polynomial rate faster than $\nu=\bar{\nu}$. 

\begin{theorem}
\label{th_jacobi-geom}
Let $C$ and $\Sigma$ satisfy (H1)-(H4). Then for any $\nu<\bar{\nu}$, $G(\Sigma)\cap D_\nu(\Sigma)=\{0\}$.    
\end{theorem}

In the limit case $\nu=\bar{\nu}$, the result in Theorem \ref{th_jacobi-geom} might not be true. For instance, if the underlying cone is strictly stable, $G(\Sigma)\cap D_\nu(\Sigma)={\rm span}\{\zeta_0\}$. However, $G(\Sigma)\cap D_\nu(\Sigma)$ may be trivial as well. This can happen, for instance, if the underlying cone is stable but not strictly stable and $\zeta_0(y)=O(|y|^{-\frac{N-2}{2}}\log|y|)$ as $|y|\to\infty$ (see  Remark $3.2$ of \cite{agudelo2024jacobi}).

\vskip 3pt

Next, we introduce a notion of  \textit{nondegeneracy} of $\Sigma$ through the operator $J_{\Sigma}$.

\begin{definition}\label{def:S_dila_nond} 
\begin{itemize}
    \item[i.] The minimal surface $\Sigma$ is dilation nondegenerate if $\bar{\nu}>2-N$ and $D_\nu(\Sigma)=\{0\}$ for any $\nu<\bar{\nu}$.

    \vskip 3pt
    \item[ii.] The minimal surface $\Sigma$ is $S$-dilation nondegenerate if $\bar{\nu}>2-N$ and $D_\nu(\Sigma,S)=\{0\}$ for any $\nu<\bar{\nu}$.\end{itemize}
\end{definition}

Roughly speaking, $\Sigma$ is $S$-dilation-nondegenerate if the Jacobi field $\zeta_0$ does not decay {\it too fast} and if there are no nontrivial $S$-invariant Jacobi fields which decay faster than $\zeta_0$.

\vskip 3pt
In view of Theorem \ref{th_jacobi-geom}, when $C$ is non-trivial, Definition \ref{def:S_dila_nond} is violated if and only if either $\bar{\nu}\le 2-N$ or there are nongeometric bounded Jacobi fields which decay faster than $\zeta_0$.\\ 

In general, it is a difficult task to exhibit minimal hypersurfaces $\Sigma$ which are \textit{dilation degenerate}, or in other words, that violate one of the properties required in Definition \ref{def:S_dila_nond}. This suggest that Definition \ref{def:S_dila_nond} might be an appropriate tool in the classification of surfaces $\Sigma$ satisfying (H1)-(H4).\\

\vskip 3pt
\begin{remark} Assume that $\bar{\nu}>2-N$ and let $S$ be a subgroup \blue{of} $O(N+1)$. If $\Sigma$ is $S-$nondegenerate in the sense introduced in Definitions 1.3 and 1.4 in \cite{agudelo2024jacobi}, then $\Sigma$ is also $S-$dilation nondegenerate.
\end{remark}

When the underlying cone $C$ is trivial, we are able to construct an example of dilation degenerate minimal hypersurface asymptotic to $C$ in the sense of (H3). However, the complement of such a hypersurface in $\R^{N+1}$ is connected. 
\begin{theorem}\label{the-deg-hyp}
Let $N\ge 3$. Then there exists a dilation-degenerate minimal hypersurface $\Sigma$ asymptotic to $\R^N\times\{0\}$ such that $\R^{N+1}\setminus\Sigma$ is connected. 
\end{theorem}

\vskip 3pt

\color{black}Now we focus on developing a linear theory to solve the Jacobi equation
\begin{equation}
\label{eq-Jacobi}
J_\Sigma\phi=f\qquad\text{in }\Sigma
\end{equation}
under suitable assumptions on $f$. 

\vskip 3pt
We start by introducing  the functional analytic setting we use for solving equation \eqref{eq-Jacobi}. Let $\nu \in \R$ and $S$ be a subgroup of $O(N+1)$. The space $C^{0,\alpha}_{\nu}(\Sigma,S)$ denotes the space of $S$-invariant functions $f\in C^{0,\alpha}_{loc}(\Sigma)$ with finite norm
\begin{equation}\notag
\|f\|_{C^{0,\alpha}_{\nu}(\Sigma)}:=\sup_{y\in\Sigma}(|y|+1)^{-\nu}\|f\|_{C^{0,\alpha}(B_1(y))},
\end{equation}
where $B_1(y)\subset \Sigma$ is the geodesic ball in $\Sigma$ of radius one and centered at $y\in \Sigma$.

\vskip 3pt

In what follows, we use the notation 
$$\|f\|_{\infty,\nu}:=\|(|y|+1)^{-\nu}f\|_\infty.
$$

The space $C^{2,\alpha}_{\nu}(\Sigma,S)$ denotes the space of $S$-invariant functions $\phi\in C^{2,\alpha}_{loc}(\Sigma)$ with finite norm
\begin{equation}\notag
\|\phi\|_{C^{2,\alpha}_\nu(\Sigma)}:=\sup_{y\in\Sigma}(|y|+1)^{-\nu+2}\|\nabla^2_\Sigma \phi\|_{C^{0,\alpha}(B_1(y))}
+\|\nabla_\Sigma \phi\|_{\infty,\nu-1}+\|\phi\|_{\infty,\nu}.
\end{equation}

Recall that $\Gamma\subset \mathbb{S}^{N}$ is a smooth minimal submanifold of $\mathbb{S}^N$. Let $\{\lambda_j\}_j$ be the sequence of eigenvalues of the operator $-J_{\Gamma}$ in $L^2(\Gamma)$, where $J_{\Gamma}:=\Delta_{\Gamma}+ |A_{\Gamma}|^2$ and $\Delta_{\Gamma}$ and $|A_{\Gamma}|$ correspond to the Laplace-Beltrami operator and the norm of the second fundamental form in $\Gamma$, respectively. It is known that 
\begin{equation}
\label{ineq:eigval_Gamma}
\lambda_0<\lambda_1\le\dots\le \lambda_j\le \dots,\qquad\lambda_j\to\infty\qquad j\to\infty.  
\end{equation}

Recalling the notations in \cite{agudelo2022k}, we set
$$j_0:=\min\left\{j\ge 0:\,\left(\frac{N-2}{2}\right)^2+\lambda_j\ge 0\right\}.$$

The authors in \cite{HS} proved that $j_0=0$ if and only if the cone $C$ is stable. For $j\ge j_0$, we set
\begin{equation}
\label{def:discrim}\Lambda_j:=\sqrt{\left(\frac{N-2}{2}\right)^2+\lambda_j}\ge 0.    
\end{equation}

If $C$ is not stable, for 
$0\leq j< j_0$, we set $$\Lambda_j:=i\sqrt{-\left(\left(\frac{N-2}{2}\right)^2+\lambda_j\right)}.$$
The numbers $- \frac{N-2}{2}\pm{\rm Re}(\Lambda_j)$ are known as the \textit{indicial roots} of $C$.

\vskip 3pt
Our second main result is concerned with the {\it solvability of the Jacobi equation} over $\Sigma$.

\begin{theorem}
\label{prop-right-inverse-Jacobi}
Assume (H1)-(H4). In addition, let $\Sigma$ be $S$-dilation-nondegenerate and assume that $\nu>2-N-\bar{\nu}$ is not an indicial root of $C$. Then, for any $f\in C^{0,\alpha}_{\nu-2}(\Sigma,S)$, there exist a solution $\phi:=\Phi(f)\in C^{2,\alpha}_{\nu}(\Sigma,S)$ to the Jacobi equation \eqref{eq-Jacobi} such that 
\begin{equation}
\label{est-Jacobi-eq-lin}
\|\phi\|_{C^{2,\alpha}_{\nu}(\Sigma)}\le c\|f\|_{C^{0,\alpha}_{\nu-2}(\Sigma)},  
\end{equation}
for some constant $c>0$ depending only of $\Sigma,\nu$.
\end{theorem}

Next, we recall a result concerned with minimal hypersurfaces $\Sigma$ asymptotic to minimal cones. As we will see, these hypersurfaces $\Sigma$ satisfy our notion of nondegeneracy and hence for them Theorem \ref{prop-right-inverse-Jacobi} holds true.

\begin{theorem}[\cite{HS}]
\label{th-Sigma}
Let $N\ge 7$. Assume that $C\subset\R^{N+1}$ is a minimising cone. Then there exist exactly two oriented, embedded smooth area minimising hypersurfaces $\Sigma^\pm\subset\R^{N+1}$ such that
\begin{enumerate}
\item $\Sigma^\pm$ are asymptotic to $C$ at infinity and do not intersect $C$.
\item $dist(\Sigma^\pm,\{0\})=1$.
\end{enumerate}
\end{theorem}
For $N+1=m+n$ and $\xi\in\R^{N+1}$ we write $\xi=(x,y)\in\R^m\times\R^n$ and we consider the Lawson's Cone
$$
C_{m,n}:=
\{(x,y)\in\R^m\times\R^n:\,(n-1)|x|^2=(m-1)|y|^2\},\qquad m,\,n\ge 2.
$$
The Lawson cone is a minimal hypersurface, which is invariant under the action of the group of orthogonal transformations $S=O(m)\times O(n)$. In the case $m=n$, $C_{m,m}$ is known as the Simon's cone.

We remark that for $m,n \geq 2$, the open set $\R^{N+1}\backslash C_{m,n}$ has two connected components, namely
$$
E_{m,n}^\pm:=\{(x,y)\in\R^m\times\R^n:\,\pm((n-1)|x|^2-(m-1)|y|^2)<0\}
$$
corresponding to the interior and the exterior of $C_{m,n}$, respectively. The sets $E_{m,n}^\pm$ help us describe an important feature of $C_{m,n}$ that has been already studied and used in \cite{agudelo2022doubling,agudelo2022k,agudelo2024jacobi,alencar2005m,HS,Mazet_2017} and that is summarised in the next results. 

\begin{theorem}\cite{Mazet_2017}
\label{th-Al}
Let $m,\,n\ge 2$ and set $n+m=N+1\ge 8$. Then there exist exactly two smooth minimal hypersurfaces $\Sigma^\pm_{m,n}\subset E^\pm_{m,n}$ which are asymptotic to $C_{m,n}$ at infinity and ${\rm dist}(\Sigma^\pm_{m,n},0)=1$. Moreover, $\Sigma^\pm_{m,n}$ are $O(m)\times O(n)$-invariant and the Jacobi field $y\cdotp\nu_{\Sigma^\pm_{m,n}}(y)$ never vanishes.
\label{th_min-surf_cone}
\end{theorem}
In dimension $4\le m+n=N+1\le 7$, we have the following existence result about minimal hypersurfaces.

\begin{theorem}[\cite{alencar2005m, Mazet_2017}]
\label{th-Sigma-low-dim}
Let $m,n\ge 2$, $m+n\le 7$. Then there exists a unique complete, embedded, $O(m)\times O(n)$ invariant minimal hypersurface $\Sigma_{m,n}$ such that 
\begin{enumerate}
\item \label{Sigma-as} $\Sigma_{m,n}$ is asymptotic to the cone $C_{m,n}$ at infinity,
\item \label{Sigma-osc} $\Sigma_{m,n}$ intersects $C_{m,n}$ infinitely many times,
\item \label{Sigma-ort} $\Sigma_{m,n}$ meets $\R^m\times\{0\}$ orthogonally,
\item \label{Sigma-dist} $dist(\Sigma_{m,n}, 0)=1$.
\end{enumerate}
\end{theorem}

Our next main result is concerned with examples of $S-$dilation nondegenerate minimal surfaces.

\begin{theorem}\label{theo:examples} 
    The hypersurfaces constructed in Theorems \ref{th-Sigma} and \ref{th-Al} are dilation non-degenerate. The hypersurfaces constructed in Theorem \ref{th-Sigma-low-dim} are $S-$dilation nondegenerate with $S=O(m)\times O(n)$.
\end{theorem}

Next, we apply our linear Theory to solve the equation
\begin{equation}
\label{eq-trA3}
J_\Sigma \phi={\rm tr}(A_\Sigma^3)\qquad\text{in }\Sigma,
\end{equation}.

In a subsequent work, we will use Theorem \ref{prop-right-inverse-Jacobi} and the fact that equation \eqref{eq-trA3} naturally in the construction of solutions to the Allen-Cahn equation

\begin{equation}\label{allen-cahn}
-\Delta u=u-u^3\qquad\text{in }\R^N.
\end{equation}
to provide a general construction of stable and unstable solutions of the Allen-Cahn equation with nodal sets closed to minimal cones. Note that, the hypersurface $\Sigma$ constructed in Theorem \ref{the-deg-hyp} is not useful for the construction of solutions to the Allen-Cahn equation \eqref{allen-cahn}, since $\R^{N+1}\setminus\Sigma$ is connected.\\

\color{black}The links between equation \eqref{eq-trA3} and the Allen-Cahn equation can be seen in several papers about the construction of solutions to the Allen-Cahn equation (see for instance \cite{agudelo2022doubling,agudelo2022k,del2011giorgi,del2013entire}).\\

Finally, the decay rate of the solution $\phi$ predicted by our general theory is not always sharp. In fact, using the symmetries present  in the hypersurfaces $\Sigma$ in Theorems \ref{th-Al} or \ref{th-Sigma-low-dim}, we can reduce equation \eqref{eq-trA3} to an ODE. In case $N=3,\,4$, the ODE analysis gives a better estimate of the decay of the solution at infinity.\\

The plan of the paper is the following: in Section $2$ we will introduce the suitable functional setting and the Fredholm theory to study the Jacobi equation $J_\Sigma\phi=f$ in a general minimal hypersurface $\Sigma$ which is asymptotic to a nontrivial cone at infinity. In Section $3$ we will prove Theorem \ref{th_jacobi-geom} and use the $S$-dilation-non-degeneracy to solve such an equation. In Section $4$ we prove Theorems \ref{the-deg-hyp} and \ref{theo:examples}, that is we prove that the hypersurfaces constructed in Theorems \ref{th-Sigma}, \ref{th-Al} and \ref{th-Sigma-low-dim} fulfill our notion of nondegeneracy. Moreover, we treat equation \eqref{eq-trA3}, at least for some of these hypersurfaces. \color{black}In section $5$ we will discuss the ODE analysis for $O(m)\times O(n)$-invariant hypersurfaces and the improvement of decay when it is possible.

\section{Fredholm Theory for the Jacobi equation}

Consider the {\it Jacobi equation} $J_{\Sigma}\phi =f$ in $\Sigma$, where $f:\Sigma \to \R$ is given. Let us solve for $\phi$ in the form 
 \begin{equation}\label{eqn:changevar_phitou}
 \phi(y):=|y|^{-\frac{N-2}{2}}u(y) \quad \hbox{for}\quad y\in \Sigma\setminus\{0\}.\end{equation}
 
From Lemma 3.1 in \cite{agudelo2024jacobi}, the equation for $u$ reads as 
\begin{equation}\label{eqn:Jac_transf}
\mathcal{L}u=g\qquad\text{in $\Sigma\setminus \{0\}$,}
\end{equation}
where 
\begin{equation}\label{eqn:changvar_Jacobieqn}
\mathcal{L}u:=|y|^N\diver(|y|^{2-N}\nabla_\Sigma u)+
\left(|y|^2|A_\Sigma|^2+|y|^{\frac{N+2}{2}}\Delta_\Sigma(|y|^{-\frac{N-2}{2}})\right)u    
\end{equation}
and
$$
g(y):=|y|^{\frac{N+2}{2}}f(y).
$$

The strategy to solve the Jacobi equation for $\phi$ is motivated by the developments in \cite{P} and those in section 3 in \cite{agudelo2024jacobi}. The idea consists, first on solving the equation \eqref{eqn:Jac_transf} for $u$ and then translate the properties to the corresponding solution $\phi$ by pulling back the change of variable \eqref{eqn:changevar_phitou}. We begin by introducing the functional analytic setting with which we work.\\

Consider the space
$$
L^2(\Sigma,|y|^{-N}):=\{v\in L^2_{loc}(\Sigma)\,:\,|y|^{\frac{-N}{2}}v\in L^2(\Sigma)\}.
$$

Observe that $L^2(\Sigma,|y|^{-N})$ is a Hilbert space with the scalar product 
\begin{equation}
\label{def-duality-prod}
\langle v,u\rangle:=\int_\Sigma v u |y|^{-N} d\sigma, \qquad  v,u\in L^2(\Sigma,|y|^{-N}).
\end{equation}

Next, for $\delta\in\R$ arbitrary but fixed we set $\Gamma_{\delta}(y):=(1 +|y|^2)^\frac{\delta}{2}$ for $y\in \Sigma$. Observe that $\Gamma_{\delta}:\Sigma \to \R$ is a smooth function such that $\Gamma_{\delta}^{-1}=\Gamma_{-\delta}$ and $\Gamma_0=1$. 

Consider also the space 
$$
L^2_{\delta}(\Sigma,|y|^{-N}):=\{v\in L^2_{loc}(\Sigma)\,:\, \Gamma_{-\delta}|y|^{\frac{-N}{2}}v\in L^2(\Sigma)\}
$$
with the inner product
\begin{equation}
\label{def-scalar-prod}
\langle v,u\rangle_{\delta}:=\int_\Sigma \Gamma_{-2\delta} v u |y|^{-N} d\sigma, \qquad  v,u\in L^2_{\delta}(\Sigma,|y|^{-N})
\end{equation}
and the norm $\|\cdot\|^2_{L^2_{\delta}(\Sigma,|y|^{-N})}:=\langle\cdot,\cdot \rangle_{\delta}$. Observe that $L^2(\Sigma,|y|^{-N})=L^2_0(\Sigma,|y|^{-N})$.\\

In what follows, $X'$ denotes the topological dual space of a given Banach space $X$. We now verify that the dual space of $L^2_\delta(\Sigma,|y|^{-N})$ is the space $L^2_{-\delta}(\Sigma,|y|^{-N})$. 

\begin{lemma}
\label{lemma-dual}
Let $\delta \in \R$. In the above notations, $(L^2_\delta(\Sigma,|y|^{-N}))'=L^2_{-\delta}(\Sigma,|y|^{-N}))$ and the duality product is given by \eqref{def-duality-prod}.   
\end{lemma}

\begin{proof}
Let $\psi\in \big(L^2_\delta(\Sigma,|y|^{-N})\big)'$ and consider  $l_{\psi}\in \big(L^2_0(\Sigma,|y|^{-N})\big)'$ defined by
$<l_{\psi},v>:= \psi(v\Gamma_\delta)$ for $v\in L^2_0(\Sigma,|y|^{-N})$. The Riesz representation Theorem implies that there exists a unique $w\in L^2_0(\Sigma,|y|^{-N})$ such that
$$\psi(v\Gamma_\delta)=<l_{\psi},v>=\langle w,v\rangle \quad \hbox{for } v\in L^2_0(\Sigma,|y|^{-N}).
$$

Set $u:=w\Gamma_{-\delta}\in L^2_{-\delta}(\Sigma,|y|^{-N})$ and for any $z\in L^2_\delta(\Sigma,|y|^{-N})$ write 
$$
v:=\Gamma_{-\delta}z\in L^2_0(\Sigma,|y|^{-N})
$$
and evaluate
$$
\psi(z)=l_{\psi}(v\Gamma_\delta)=\langle w,v\rangle=\langle w\Gamma_{-\delta},z\rangle=\langle u,z\rangle.
$$

This concludes the proof. 
\end{proof}

Let $l\in \mathbb{N}$ and consider now the {\it weighted Sobolev space}
$$
\mathcal{W}^{l,2}_{\delta}(\Sigma,|y|^{-N}):=\{w\in W^{l,2}_{loc}(\Sigma)\,:\,|y|^j\nabla_{\Sigma}^{(j)}w\in L^2_{\delta}(\Sigma,|y|^{-N}) \quad 0\leq j\leq l\}
$$
with the norm
\begin{equation}\label{def-norm_W2}
\|w\|_{\mathcal{W}^{l,2}_\delta(\Sigma,|y|^{-N})}= \Big(\sum_{j=0}^l \big\| |y|^{j}\nabla_{\Sigma}^{(j)} w\big\|^2_{L^2_{\delta}(\Sigma,|y|^{-N})}\Big)^{\frac{1}{2}}.    
\end{equation}

We introduce next the operator
$A_\delta:\,\mathcal{W}^{2,2}_\delta(\Sigma,|y|^{-N})\to L^2_\delta(\Sigma,|y|^{-N})$ defined by
\begin{equation}\label{def:operator_A_delta}
w\mapsto  A_{\delta}w:=\mathcal{L}w \quad \hbox{for}\quad w\in \mathcal{W}^{2,2}_{\delta}(\Sigma,|y|^{-N}).
\end{equation}

Observe that $A_{\delta}$ is continuous with respect to the norms in \eqref{def-norm_W2} and the one  induced by the inner product in \eqref{def-scalar-prod}. 

\vskip 3pt

However, in what follows and unless stated otherwise, we consider ${\rm Dom}(A_{\delta}):=W^{2,2}_{\delta}(\Sigma,|y|^{-N})$ endowed with the inner product defined in \eqref{def-scalar-prod}. Since ${\rm Dom}(A_{\delta})$ dense in $L^2_{\delta}(\Sigma,|y|^{-N})$, the operator $A_{\delta}$ is an unbounded operator in $L^2_{\delta}(\Sigma,|y|^{-N})$. 

\vskip 3pt
Next, we consider the adjoint operator of $A_{\delta}$, namely $A_{\delta}^*:{\rm Dom}(A_{\delta}^*)\to L^2_{-\delta}(\Sigma,|y|^{-N})$ defined by the duality
$$
\langle A_{\delta}^*v,u\rangle:=\langle v,A_{\delta} u\rangle \quad \hbox{for } v\in {\rm Dom}(A_{\delta}^*),\,\, u\in W^{2,2}_{\delta}(\Sigma,|y|^{-N}).
$$

In view of Lemma \ref{lemma-dual}, \begin{equation}\notag
\begin{aligned}
{\rm Dom}(A^\star_\delta)=\{v\in& L^2_{-\delta}(\Sigma,|y|^{-N}):\\
&\exists\,z\in L^2_{-\delta}(\Sigma,|y|^{-N}): \forall\,u\in\mathcal{W}^{2,2}_\delta(\Sigma,|y|^{-N}):
\langle \mathcal{L} u,v\rangle=\langle u,z\rangle\}.
\end{aligned}
\end{equation}

This is in essence the functional analytic context in \cite{P}.

\begin{lemma}\label{lemma:adjoint_oper}
Let $\delta\in\R$. In the previous notations, the adjoint of $A_\delta$ is given by $A_\delta^\star=A_{-\delta}$.
\end{lemma}
\begin{proof}
It is enough to show that ${\rm Dom}(A_\delta^*)=\mathcal{W}^{2,2}_{-\delta}(\Sigma,|y|^{-N})$. In order to do so, we notice that the symmetry of $\mathcal{L}$ yields $\mathcal{W}^{2,2}_{-\delta}(\Sigma,|y|^{-N})\subset {\rm Dom}(A^\star_\delta)$.

\vskip 3pt
To prove the opposite inclusion we proceed as follows. Let $g$ be the metric tensor of $\Sigma$. First, using the fact that $g$ is smooth and positive on $\Sigma$, for any $R>0$, the restriction of $g$ to $K_R$ (see \eqref{eqn:split_Sigma}), namely $g_{K_R}$, is equivalent to the metric of the cylinder $(0,R)\times \Gamma$. Thus, standard elliptic estimates hold on $K_R$. 

\vskip 3pt
On the other hand, it follows from \eqref{Sigma-normal-graph} that
\begin{equation}\label{eqn:asympt_metric_Sigma_EmdenFowler}
g=e^{2t} (1 +o(1))(dt^2 + d\theta^2) \quad \hbox{as} \quad t\to + \infty,   
\end{equation}
where for $r\theta \in C$ we have written $r=e^{t}$ for $t$ large and where $d\theta^2$ denotes the metric tensor on the compact manifold $\Gamma$. 

\vskip 3pt
Now, given $v\in {\rm Dom}(A^\star_\delta)$, $v\in L^2_{-\delta}(\Sigma,|y|^{-N})$ and there exists $z\in L^2_{-\delta}(\Sigma,|y|^{-N})$ such that $v$ is a distributional solution to the equation $\mathcal{L}v=z$ in $\Sigma$. In turn, following the Remark 3.1 in \cite{agudelo2024jacobi}, writing $v(y)=\tilde{v}(t,\theta)$ for $y=e^t\theta + w(e^t,\theta)\nu_C(\theta)$ for $t>0$ large and $\theta \in \Gamma$, we have
\begin{equation}
\label{eq-SigmaR}
\mathcal{L} v = \partial_{tt} \tilde{v} + \Delta_{\Gamma} \tilde{v} + \Big(|A_{\Sigma}|^2 - \big(\frac{N-2}{2}\big)^2\Big)\tilde{v} + \mathcal{R}(\tilde{v}) \quad \hbox{in }\Sigma_{R}, 
\end{equation}
where 
\begin{equation}
\label{est-remainder}
|\mathcal{R}(\tilde{v})|\leq c e^{-\eta t}(|\tilde{v}|+|\partial_t\tilde{v}|+|\nabla_\Gamma\tilde{v}|+|\partial_{tt}\tilde{v}|+|\nabla^2_\Gamma\tilde{v}|+|\partial_t\nabla_\Gamma\tilde{v}|)\qquad\forall\,(t,\theta)\in(t_0,\infty)\times\Gamma,
\end{equation}
for some $c,t_0>0$ large enough and some $\eta>0$ small enough. 
\vskip 3pt

We then follow the lines of the proof of Proposition 8.3.1 in \cite{P} with only slight changes to find that $v\in\mathcal{W}^{2,2}_{-\delta}(\Sigma,|y|^{-N})$. Since $v$ is arbitrary, ${\rm Dom}(A^*_{\delta})\subset \mathcal{W}^{2,2,}_{-\delta}(\Sigma, |y|^{-N})$ and this completes the proof of the lemma.
\end{proof}

Recall that $\Gamma\subset \mathbb{S}^{N}$ is a smooth minimal submanifold of $\mathbb{S}^N$. Let $\{\lambda_j\}_j$ be the sequence of eigenvalues of the operator $-J_{\Gamma}$ in $L^2(\Gamma)$, where $J_{\Gamma}:=\Delta_{\Gamma}+ |A_{\Gamma}|^2$ and $\Delta_{\Gamma}$ and $|A_{\Gamma}|$ correspond to the Laplace-Beltrami operator and the norm of the second fundamental form in $\Gamma$, respectively.\\

 We recall that the eigenvalues $\lambda_j$ of the Jacobi operator $-J_\Gamma:=-(\Delta_\Gamma+|A_\Gamma|^2)$ of $\Gamma$ satisfy 
\begin{equation}
\label{ineq:eigval_Gamma}
\lambda_0<\lambda_1\le\dots\le \lambda_j\le \dots,\qquad\lambda_j\to\infty\qquad j\to\infty.  
\end{equation}

We also recall that the numbers ${-\frac{N-2}{2}\pm\rm Re}(\Lambda_j)$ are the \textit{indicial roots} of $C$.
\begin{remark}\label{rmk:ellipt_estimates}
Following the lines of the proof of the Lemma \ref{lemma:adjoint_oper} above, and the arguments of the proof of the Propositions 8.3.1 and 8.4.1 in \cite{P}, with only minor changes, we obtain that given any $\delta \in \R\setminus \{\pm {\rm Re}(\Lambda_j)\}_{j\geq 0}$:
\vskip 3pt
 there exists $R_0>0$ such that for any $R>R_0$, there exists $c=c(\delta,R)>0$ such that for any $u\in \mathcal{W}^{2,2}_{\delta}(\Sigma,|y|^{-N})$,
\begin{equation}\label{ineq:ellipt_estimates_compact}
\|u\|_{\mathcal{W}^{2,2}_{\delta}(\Sigma,|y|^{-N})}\leq c \Big(\|{A}_{\delta}u\|_{L^{2}_{\delta}(\Sigma,|y|^{-N})} + \| u\|_{L^{2}(K_R)}\Big).
\end{equation}
\end{remark}

For $\delta \in \R$, we write $L^2_{\delta}(\Sigma,|y|^{-N})=K_{\delta}\oplus K_{\delta}^{\perp}$, where
\begin{equation}\label{def:Kernel_delta}
K_\delta\,:=\,{\rm Ker}(A_\delta)\,=\,\{u\in \mathcal{W}^{2,2}_\delta(\Sigma,|y|^{-N}):\,\mathcal{L}u=0\} 
\end{equation}
and $K_{\delta}^{\perp}$ denotes the orthogonal complement of $K_{\delta}$ with respect to the inner product in $L^2_{\delta}(\Sigma,|y|^{-N})$, i.e.,
$$
K_{\delta}^{\perp}=\big\{u\in L^2_{\delta}(\Sigma,|y|^{-N})\,:\,  \hbox{for all }v\in K_{\delta},\,\,\int_{\Sigma}\Gamma_{-2\delta} uv|y|^{-N}d\sigma =0\big\}.
$$

\begin{lemma}\label{lemma:invert_Adelta}
Let $\delta \in \R\setminus \{\pm {\rm Re}(\Lambda_j)\}_{j\geq 0}$ and $R_0>0$ be as in the Remark \ref{rmk:ellipt_estimates}. For any $R>R_0$, there exists $C=C(\delta,R)>0$ such that for any $u\in K_{\delta}^{\perp}\cap \mathcal{W}^{2,2}_{\delta}(\Sigma,|y|^{-N})$,
\begin{equation}\label{ineq:estima_inject}\|u\|_{L^2(K_{R})}\leq C\|A_{\delta} u\|_{L^2_{\delta}(\Sigma,|y|^{-N})}.
\end{equation}
\end{lemma}

\begin{proof} Assume that $R>R_0$. We prove \eqref{ineq:estima_inject} arguing by contradiction. Let $\{u_k\}_k\in K_{\delta}^{\perp}\cap W^{2,2}_{\delta}(\Sigma,|y|^{-N})$ be a sequence such that such that 
$$
\|A_{\delta}u_k\|_{L^2_{\delta}(\Sigma,|y|^{-N})} \to 0 \qquad  \hbox{and} \qquad  \hbox{for all } \quad k\in \mathbb{N},\,\,\|u_k\|_{L^2(K_R)}=1.
$$

We observe that \eqref{ineq:ellipt_estimates_compact} is true. Consequently, $\{u_k\}_k$ is bounded in $W^{2,2}_{\delta}(\Sigma,|y|^{-N})$. Passing to a subsequence, if necessary, we may assume that $u_k\rightharpoonup u$ weakly in  $W^{2,2}_{\delta}(\Sigma,|y|^{-N})$. The weak convergence of $\{u_k\}_k$ and the fact that  $\{u_k\}_k\subset K_{\delta}^{\perp}$ imply that $u\in K_{\delta}^{\perp}$.

\vskip 3pt 

We may also assume that $u_k\rightharpoonup u$ weakly in $W^{2,2}_{loc}(\Sigma)$, which in turn yields that $\mathcal{L}u=0$ a.e. in $\Sigma$. We conclude that $u\in W^{2,2}_{\delta}(\Sigma,|y|^{-N})$ is such that $A_{\delta}u=0$, i.e., $u\in K_{\delta}$. Therefore, $u\in K_{\delta}\cap K_{\delta}^{\perp}=\{0\}$.

\vskip 3pt
However, the Rellich-Kondrachov Theorem yields that $u_k \to u$ strongly in $L^2(K_R)$ and hence $\|u\|_{L^2(K_R)}=1$, which is a contradiction. This proves the lemma.
\end{proof}

\begin{proposition}
\label{propo-Fredholm}
Let $\delta\in\R\setminus\{\pm{\rm Re}(\Lambda_j)\}_{j\ge 0}$. Then, $A_\delta:W^{2,2}_{\delta}(\Sigma,|y|^{-N})\to L^2_{\delta}(\Sigma,|y|^{-N})$ is a Fredholm operator, i.e., 
\begin{itemize}
    \item[i.] ${\rm Im}(A_{\delta}) $ is closed and
    \item[ii.] ${\rm dim}(K_{\delta})$ and ${\rm codim}({\rm Im}(A_{\delta}))$ are finite.
\end{itemize}

{Moreover, the same conclusion holds for $A_{-\delta}:W^{2,2}_{-\delta}(\Sigma,|y|^{-N})\to L^2_{-\delta}(\Sigma,|y|^{-N})$ and ${\rm codim}({\rm Im}(A_{\delta}))={\rm dim}(K_{-\delta})$.} 
 \end{proposition}

\begin{proof}
 First, the proof that ${\rm dim}(K_{\delta})$ is finite is done arguing by contradiction following the arguments of the proof of Proposition 9.1.1 in \cite{P} making use of the estimate \eqref{ineq:ellipt_estimates_compact}.\\

Next, we prove that ${\rm Im}(A_{\delta})$ is closed in $L^2_{\delta}(\Sigma,|y|^{-N})$. We argue as in the proof of Theorem 9.2.1 in \cite{P}. Let $\{g_k\}_k\in L^2_{\delta}(\Sigma,|y|^{-N})$ and $\{u_k\}_k\in W^{2,2}_{\delta}(\Sigma,|y|^{-N})$ be such that for any $k\in \mathbb{N}$, $A_{\delta}u_k=g_k$ and $g_k\to g$ strongly in $L^2_{\delta}(\Sigma,|y|^{-N}) $. By projecting each $u_k$ onto $K_{\delta}^{\perp}$, we may assume that $\{u_k\}_k\in K_{\delta}^{\perp}\cap W^{2,2}_{\delta}(\Sigma,|y|^{-N})$ so that for any $k\in \mathbb{N}$ and for any $v\in K_{\delta}$,
\begin{equation}
\label{ort-Sigma-delta}
\langle u_k,v\rangle_\delta=0.
\end{equation}

The estimate \eqref{ineq:ellipt_estimates_compact} in Remark \ref{rmk:ellipt_estimates} and the estimate \eqref{ineq:estima_inject} in Lemma \ref{lemma:invert_Adelta} yield that for some constant $C>0$ independent of $k$, and for any $k_1,k_2\in \mathbb{N}$,
$$
\|u_{k_1}-u_{k_2}\|_{W^{2,2}_{\delta}(\Sigma,|y|^{-N})}\leq C\|g_{k_1}-g_{k_2}\|_{L^{2}_{\delta}(\Sigma,|y|^{-N})}.
$$

Since $\{g_k\}_k$ is a Cauchy sequence in $L^2_{\delta}(\Sigma,|y|^{-N})$ and $W^{2,2}_{\delta}(\Sigma,|y|^{-N})$ is Banach, there exists $u\in W^{2,2}_{\delta}(\Sigma,|y|^{-N})$
 such that $u_k \to u$ strongly in $W^{2,2}_{\delta}(\Sigma,|y|^{-N})$. We thus conclude that $A_{\delta}u=g$ and from \eqref{ort-Sigma-delta} that $u\in K_{\delta}^{\perp}$. This proves the claim. 

\vskip 3pt
At this point, we remark that $A_{-\delta}:W^{2,2}_{-\delta}(\Sigma,|y|^{-N})\to L^2_{-\delta}(\Sigma,|y|^{-N})$ is such that ${\rm dim}(K_{-\delta})$ is finite and ${\rm Im}(A_{-\delta})$ is closed in $L^2_{-\delta}(\Sigma,|y|^{-N})$.  

\vskip 3pt
Now we prove that ${\rm codim}({\rm Im}(A_{\delta}))={\rm dim}(K_{-\delta})$. We proceed as follows. From Lemma \ref{lemma:adjoint_oper} that for any $u\in W^{2,2}_{\delta}(\Sigma,|y|^{-N})$ and for any $v\in W^{2,2}_{-\delta}(\Sigma,|y|^{-N})$,
$$
\langle A_{\delta}u,v\rangle=\langle u,A_{-\delta} v\rangle.
$$

Observe also that 
$$
L^{2}_{\delta}(\Sigma,|y|^{-N})={\rm Im}(A_{\delta})\oplus{\rm Im}(A_{\delta})^{\perp} 
$$
and 
$$
L^{2}_{-\delta}(\Sigma,|y|^{-N})=K_{\delta}\oplus K_{-\delta}^{\perp}
$$
where the orthogonal complements are considered with respect to the inner products $\langle\cdot ,\cdot\rangle_{\delta}$ and $\langle \cdot,\cdot\rangle_{-\delta}$, respectively. Using the notations of the proof of Lemma \ref{lemma-dual}, given $g,u\in W^{2,2}_{\delta}(\Sigma,|y|^{-N})$, and setting {$\tilde{\rm g}:= \Gamma_{-2\delta}\,g\in W^{2,2}_{-\delta}(\Sigma,|y|^{-N})$} and {$\tilde{u}:=\Gamma_{-2\delta}\,u\in W^{2,2}_{-\delta}(\Sigma,|y|^{-N})$}, we have that
$$
\langle g ,A_{\delta}u\rangle_{\delta}=\langle \tilde{g} ,A_{\delta}u\rangle=\langle A_{-\delta}\tilde{g},u\rangle = \langle A_{-\delta}\tilde{g},\tilde{u}\rangle_{-\delta}. 
$$

From the previous identities, using a density argument, the linear isometry $g \in L^2_{\delta}(\Sigma,|y|^{-N})\mapsto \tilde{g}:=\Gamma_{-2\delta}g\in L^2_{-\delta}(\Sigma,|y|^{-N})$ maps ${\rm Im}(A_{\delta})^{\perp}$ bijectively onto $K_{-\delta}$. This proves the claim and concludes the proof of the proposition.\end{proof}

\begin{remark}\label{rmk:invert_A_delta}
Observe that $K_{\delta}^{\perp}\cap \mathcal{W}^{2,2}_{\delta}(\Sigma,|y|^{-N})$ is closed in $\mathcal{W}^{2,2}_{\delta}(\Sigma,|y|^{-N})$ and ${\rm Im}(A_{\delta})$ is closed in $L^{2}_{\delta}(\Sigma,|y|^{-N})$. Also, the estimates in \eqref{ineq:ellipt_estimates_compact} and \eqref{ineq:estima_inject} imply that $A_{\delta}:K_{\delta}^{\perp} \cap \mathcal{W}^{2,2}_{\delta}(\Sigma,|y|^{-N})\to {\rm Im}(A_{\delta})$ is bijective with continuous inverse. 
 
 \vskip 3pt
 The open mapping Theorem implies that $A_{\delta}:K_{\delta}^{\perp}\cap W^{2,2}_{\delta}(\Sigma,|y|^{-N}) \to {\rm Im}(A_{\delta})$ is a linear topological isomorphism.
\end{remark}

Let $\delta\notin\{\pm{\rm Re}(\Lambda_j)\}_{j\ge 0}$. Set $d(\delta):=\dim(K_\delta)$, which is finite in view of Proposition \ref{propo-Fredholm}. The set $\{v_{\delta,j}\}_{1\le j\le d(\delta)}$ will denote an orthonormal basis (with respect to the scalar product in \eqref{def-scalar-prod}) of $K_{\delta}$.

In the following proposition, the constant $c$ denotes a universal constant.

\begin{proposition}
\label{prop-right-inverse}
Let $\delta\in \R\backslash\{\pm{\rm Re}(\Lambda_j)\}_{j\ge 0}$. 
\begin{itemize}
\item[i.] Assume that $K_{-\delta}=\{0\}$. Then, for any $g\in L^2_\delta(\Sigma,|y|^{-N})$, there exists a unique $u\in \mathcal{W}^{2,2}_\delta(\Sigma,|y|^{-N})\cap K_{\delta}^{\perp}$ such that $\mathcal{L}u=g$ and fulfilling
\begin{equation}
\label{est-u-injective}
\|u\|_{\mathcal{W}^{2,2}_\delta(\Sigma,|y|^{-N})}\le c\|g\|_{L^{2}_\delta(\Sigma,|y|^{-N})}.
\end{equation}

\item[ii.] Assume that $K_{-\delta}\ne\{0\}$. Then, for any $g\in L^2_\delta(\Sigma,|y|^{-N})$, there exists a unique $u\in \mathcal{W}^{2,2}_\delta(\Sigma,|y|^{-N})\cap K_{\delta}^{\perp}$ and uniquely determined scalars $\mu_1,\dots,\mu_{d(-\delta)}\in \R$ such that 
\begin{equation}
\label{right-inv-L}
\mathcal{L}u=g-\sum_{j=1}^{d(-\delta)} \mu_jv_{-\delta,j}\Gamma_\delta^2\quad\hbox{in} \quad \Sigma
\end{equation}
and for every $1\le j\le d(-\delta)$,
\begin{equation}\label{eqn:scalar_prprojec}
\mu_j:=\int_{\Sigma} \,gv_{-\delta,j}  \,|y|^{-N}d\sigma.
\end{equation}

Moreover, 
\begin{equation}
\label{est-u-non-injective}
\|u\|_{\mathcal{W}^{2,2}_\delta(\Sigma,|y|^{-N})}+\sum_{j=1}^{d(-\delta)}|\mu_j|\le c\|g\|_{L^{2}_\delta(\Sigma,|y|^{-N})}.
\end{equation}
\end{itemize}
\end{proposition}
\begin{proof} We first prove {\it i.}. From Proposition \ref{propo-Fredholm}, ${\rm Im}(A_{\delta})=L^2_{\delta}(\Sigma,|y|^{-N})$. In view of Remark \ref{rmk:invert_A_delta}, Lemma \ref{lemma:invert_Adelta} and  Remark \ref{rmk:ellipt_estimates} the conclusion follows.

\vskip 3pt
Next, we prove {\it ii.}. We proceed as follows. Let $g\in L^2_{\delta}(\Sigma,|y|^{-N})$. Following the conventions of the proof of Proposition \ref{propo-Fredholm}, write $g={\rm g}+{\rm g}^{\perp}$, where ${\rm g}\in {\rm Im}(A_{\delta})$ and ${\rm g}^{\perp}\in {\rm Im}(A_{\delta})^{\perp}$. Let $\tilde{\rm g}:=\Gamma_{-2\delta}{\rm g}^{\perp}$ and observe that $\tilde{\rm g}\in K_{-\delta}$. Thus, for scalars $\mu_1,\ldots, \mu_{d(-\delta)}$, uniquely determined by \eqref{eqn:scalar_prprojec}, and using that ${\rm g}^{\perp}=
\Gamma_{2\delta}\tilde{\rm g}$,
$$
{\rm g}^{\perp}=\mu_{1}\Gamma_{2\delta}v_{-\delta,1}+\cdots+\mu_{d(-\delta)}\Gamma_{2\delta}v_{-\delta,d(-\delta)}.
$$

From this
\eqref{eqn:scalar_prprojec} follows that for any $1\leq j\leq d(-\delta)$,
\begin{equation}\label{ineq:est_scal_project}
|\mu_j|\leq \|v_{-\delta,j}\|_{L^2_{-\delta}(\Sigma,|y|^{-N})}\|g\|_{L^2_{\delta}(\Sigma,|y|^{-N})}=\|g\|_{L^2_{\delta}(\Sigma,|y|^{-N})}.
\end{equation}

Now, since ${\rm g}\in {\rm Im}(A_{\delta})$, we argue in the same manner as in the proof of part {\it i.} to find a unique $u\in \mathcal{W}^{2,2}_{\delta}(\Sigma,|y|^{-N})\cap K_{\delta}^{\perp}$ such that $A_{\delta}u={\rm g}$, i.e., $u$ and $g$ satisfy \eqref{right-inv-L}. Since ${\rm g}=g-{\rm g}^{\perp}$,
$$
\|u\|_{\mathcal{W}^{2,2}_{\delta}(\Sigma,|y|^{-N})} \leq c\|{\rm g}\|_{L^2_{\delta}(\Sigma,|y|^{-N})} \leq \|g\|_{L^2_{\delta}(\Sigma,|y|^{-N})} + \sum_{j=1}^{d(-\delta)}|\mu_j|,
$$
and using \eqref{ineq:est_scal_project}, the estimate \eqref{est-u-non-injective} follows. This concludes the proof. 
\end{proof}

\section{S-dilation nondegeneracy}

\subsection{Asymptotic behaviour of the geometric Jacobi fields}

In this subsection we discuss the asymptotic behaviour of geometric Jacobi fields.

\vskip 3pt
Recall that $\{e_1,\ldots, e_{N+1}\}$ is the standard basis in $\R^{N+1}$ and that the functions
\begin{equation}\label{def:jac_field_transl}
\zeta_j(y):=e_j\cdot \nu_\Sigma(y) \quad \hbox{for } y\in \Sigma, \,1\leq j\leq N+1,
\end{equation}
span the Jacobi fields of $\Sigma$ arising from invariances under translations.

\vskip 3pt
Recall also that
\begin{equation}\label{def:jac_field_dilat}
\zeta_0(y):=y\cdotp\nu_\Sigma(y) \quad \hbox{for }y\in \Sigma,
\end{equation}
 $\zeta_0$ generates the Jacobi fields of $\Sigma$ arising from the invariance under dilations.

\begin{lemma}
\label{lemma-dec-Jacobi}
Let $S=\{0\}$ and assume (H1)-(H4). If $\zeta \in {\rm span}\{\zeta_1,\ldots,\zeta_{N+1}\}\setminus\{0\}$, then
\begin{equation}\label{eqn:JF_transl_behav}
  \limsup_{|y|\to\infty}|\zeta(y)|>0.  
\end{equation}

Moreover, $\lim_{|y|\to\infty}\zeta_0(y)=0$ and the nontrivial Jacobi fields $\zeta$ of $\Sigma$ coming from rotations have linear growth at infinity.
\end{lemma}
\begin{proof}
First we show \eqref{eqn:JF_transl_behav}. Performing a change of basis, if necessary, we only need to prove the statement when $\zeta=\zeta_j$. Let $R>0$ be arbitrary and large, but fixed. Fix particular choices of $\nu_\Sigma$ and $\nu_C$, the unit normal vector fields of $\Sigma$ and $C$, respectively, so that any $y\in\Sigma_R$ can be written as
$$
y=r\theta+w(r,\theta)\nu_C(r\theta) \qquad \hbox{ for }r>R\, \hbox{ and }\theta\in\Gamma.
$$
It holds that
$$|\nu_\Sigma\big(r\theta+w(r,\theta)\nu_C(r\theta)\big)-\nu_C(r\theta)|\to 0\qquad r\to\infty,$$
uniformly in $\theta\in\Gamma$. Consequently, for any $1\leq j\leq N+1$, $\zeta_j=\nu_{\Sigma}\cdot e_j$ and the corresponding Jacobi field of the cone $\bar{\zeta}_j(r\theta):=\nu_C(r\theta)\cdotp e_j$ satisfy that
\begin{equation}
\label{as-Jacobi-Sigma-C}
|\zeta_j\big(r\theta+w(r,\theta)\nu_C(r\theta)\big)-\bar{\zeta}_j(r\theta)|\to 0\qquad r\to\infty
\end{equation}
uniformly in $\theta\in\Gamma$.

We claim that the Jacobi field $\bar{\zeta}_j$ of $C$ does not vanish on $C$. Once the claim is proven, then \eqref{eqn:JF_transl_behav} follows for $\zeta_j$.  

\medskip
In fact, if we assume by contradiction that there exists $j\in\{1,\dots,N+1\}$ such that $\bar{\zeta}_j=0$, then $e_j\in T_p C$ for any $p\in C\setminus{\rm sing}(C)$, where ${\rm sing}(C)$ denotes the singular set of $C$.\\

We claim that this fact yields that $C$ is a cylinder, in the sense that there exists a cone $\tilde{C}\subset\R^N$ of codimension $1$ in $\R^N$ such  that $C={\rm span}\{e_j\}\times\tilde{C}$.\\

In order to prove the claim, we fix $y\in C\setminus{\rm sing}(C)$ and a parametrisation $\varphi$ around $y$. Differentiating the relation $\nu_C(y)\cdotp e_{N+1}=0$ with respect to the local coordinates $x_i$ we can see that
$$d\nu_C(y)\partial_i\varphi\cdotp e_{N+1}=d\nu_C(y) e_{N+1}\cdotp\partial_i\varphi=0\qquad\forall\,i=1,\dots,N,$$
which gives $d\nu_C(y) e_{N+1}=0$. This yields that the principal curvature $\kappa_N$ vanishes identically, since $e_{N+1}$ is an eigenvector of $d\nu_C(y)$ with eigenvalue $0$. Since $y\in C$ is arbitrary, this proves the claim.

 
If $\tilde{C}=\R^{N-1}$, then $C=\R^N$, a contradiction with (H1). If $\tilde{C}$ is nontrivial, then the singular set of $C$ contains the straight line ${\rm span}\{e_j\}$, therefore there exists no hypersurface asymptotic to $C$, since there exists no $R>0$ such that a unit normal vector field is defined in $C\setminus B_R(0)$.\\

A similar argument shows that the Jacobi fields coming from nontrivial rotations are unbounded. In fact, let $M\in\mathcal{A}_{N+1}$ be an antisymmetric matrix (see \eqref{def:matrix_antisymm}) and assume that the Jacobi field $\zeta_M(y):=M y\cdotp \nu_\Sigma (y)$ is bounded. Then, arguing as in \eqref{as-Jacobi-Sigma-C}, together with \eqref{jacobi-fields-rot}, we see that the corresponding Jacobi field $\bar{\zeta}_M$ of $C$ is bounded. However, $\bar{\zeta}_M(r\theta)=rM\theta\cdotp \nu_C(\theta)$, and hence $\bar{\zeta}_M\equiv 0$. This yields that $C$ is invariant under the rotation $\rho\in SO(N+1)$ associated to $M$. The maximality assumption on $S$, that is (H4), yields that $\rho\in S$. Thus, from assumption (H2), we have $\rho(\Sigma)\subset\Sigma$, which in turn yields that $\zeta_M\equiv 0$. We remark that the previous argument and \eqref{as-Jacobi-Sigma-C} yield that the nontrivial Jacobi fields of $\Sigma$ coming from rotations have linear growth.\\

Finally, since the cone is homogeneous, its Jacobi field coming from dilations fulfills $\bar{\zeta}_0(r\theta):=r\theta\cdotp\nu_C(\theta)=0$, which yields that $\lim_{|y|\to\infty}\zeta_0(y)=0$.
\end{proof}



\begin{remark}
Assumptions (H1)-(H4) guarantee that the hypersurface $\Sigma$ inherits the symmetries of the cone. In other words, if $\rho\in O(N+1)$ and $\rho(C)\subset C$, then $\rho(\Sigma)\subset\Sigma$ too. This is crucial for the validity of Lemma \ref{lemma-dec-Jacobi}. In fact, let $C:=\{(x,y,z)\in\R^3:\,z^2=x^2+y^2\},\,S:=\{Id_{3}\}$ and let $M\in \mathcal{A}_3$ be the antisymmetric matrix related to rotations around the $z$-axis, that is, satisfying $M_{21}=1,\,M_{31}=M_{32}=0$. Then $$\bar{\zeta}_C(x,y,z)=\frac{1}{2z}(-y,x,0)\cdotp(-x,-y,z)=0\qquad\forall\,(x,y,z)\in C,$$
since the cone is axially symmetric, that is, invariant under rotations around the $z$-axis. However, introducing polar coordinates $(r,\vartheta)$ in $\R^2$, the graph $\Sigma$ of the function $u(x,y)=v(r,\vartheta)=r+\frac{r}{1+r^2}\cos\vartheta$ is asymptotic to $C$ but not axially symmetric. A computation shows that the corresponding function
$$\zeta_M(x,y,z)=\frac{(-y,x,0)\cdotp(-\partial_x u,-\partial_y u,1)}{\sqrt{1+|\nabla u|^2}}\to 0\qquad \text{as}\,|(x,y,z)|\to\infty$$
but it does not vanish identically on $\Sigma$.
\end{remark}

Now we can conclude the proof of Theorem \ref{th_jacobi-geom}.
\begin{proof}
This is a direct consequence of Lemma \ref{lemma-dec-Jacobi}.

\end{proof}


Now we will give a better result about the decay rate of $\zeta_0$ (see \eqref{def:rate_decay_zeta0}).
\begin{lemma}
\label{lemma-nu<0}
In the previous notations, we have $\bar{\nu}=-\frac{N-2}{2}+\bar{\delta}<0$, with $\bar{\delta}\in\{\pm{\rm Re}(\Lambda_j)\}_{j\ge 0}$.
\end{lemma}
\begin{proof}
 In order to do so, we recall that from (H3),
\begin{equation}
\label{as-zeta}
\zeta_0(y)=r^2\partial_r(r^{-1}w)(1+o(1)) 
\end{equation}
as $r\to\infty$ uniformly in $\theta\in\Gamma$, where we have used the notation
$$y=r\theta+w(r,\theta)\nu_C(\theta)\qquad\forall\,y\in\Sigma_R$$
for $R>0$ large enough. Using the fact that $\Sigma$ is asymptotic to the cone at infinity, we have $w,r\partial_r w\to 0$ as $r\to\infty$ uniformly in $\theta\in\Gamma$, so that $\zeta_0(y)\to 0$ as $|y|\to\infty$.\\

On the other hand by Lemma $11.1.4$ of \cite{P}, we know that, denoting the eigenfunctions of $-J_\Gamma$ by $\varphi_j$, we have
\begin{equation}\notag
\label{pre-as-behaviour-zeta}
    \zeta_0(y)=r^\nu(a\log r+b)\varphi_j(\theta)(1+o(1))
\end{equation}
as $r\to\infty$ uniformly in $\theta\in\Gamma$, for some $j\ge 0$ and some $\nu\in\{-\frac{N-2}{2}\pm{\rm Re}(\Lambda_j)\}\cap(-\infty,0)$, with $a=0$ if $j\ge 1$. In particular, this yields that there exists $\nu<0$ such that
$$|\zeta_0(y)|\le c(1+|y|)^\nu,$$
for some constant $c>0$, which shows that $\bar{\nu}<0$.   
\end{proof}

\subsection{$S$-Dilation-nondegeneracy and symmetry}
In this subsection we will des\-cribe the role of the $S$-dilation-nondegeneracy and symmetry in solving the Jacobi equation \ref{eq-Jacobi}. In particular, we will see that the $S$-dilation-nondegeneracy allows us to prove that the Lagrange multipliers appearing in Proposition \ref{prop-right-inverse} vanish and that, if the right-hand side $f$ is $S$-invariant, then the solution $\phi$ can be chosen to be $S$-invariant as well.\\ 
For $\delta\in\R$, we set
\begin{equation}
\begin{aligned}
L^2_\delta(\Sigma,S)&:=\{u\in L^2_\delta(\Sigma,|y|^{-N}):\, u\text{ is $S$-invariant}\},\\
W^{2,2}_\delta(\Sigma,S)&:=\{u\in \mathcal{W}^{2,2}_\delta(\Sigma,|y|^{-N}):\, u\text{ is $S$-invariant}\},\\
K_\delta(S)&:=K_\delta\cap W^{2,2}_\delta(\Sigma,S) \qquad \hbox{\big(see \eqref{def:Kernel_delta}\big)}.
\end{aligned}
\end{equation}
We set $v:=|y|^{\frac{N-2}{2}}\zeta_0$ and $\bar{\delta}:=\frac{N-2}{2} + \bar{\nu}$.
\begin{proposition}
\label{prop-non-degeneracy}
Let $\Sigma$ be $S$-dilation-nondegenerate. Then $K_\delta(S)=\{0\}$ for any $\delta<\bar{\delta}$.
\end{proposition}
\begin{proof}
Lemma $3.4$ of \cite{agudelo2024jacobi} yields that if $\delta\in\R$ and $u\in K_\delta$ then $u\in C^{2,\alpha}_{loc}(\Sigma)$ and it fulfills the pointwise estimate $$|u(y)|\le c(1+|y|)^\delta \qquad\forall\, y\in\Sigma,$$
for some constant $c>0$. As a consequence, since $\delta<\bar{\delta}$, then the corresponding Jacobi field $\phi:=|y|^{-\frac{N-2}{2}}u$ is $S$-invariant and fulfills $\phi (1+|y|)^{\frac{N-2}{2}-\delta}\in L^\infty(\Sigma)$. Thus, $\phi\in D_\nu(\Sigma,S)$ for $$\nu:=-\frac{N-2}{2}+\delta<-\frac{N-2}{2}+\bar{\delta}=\bar{\nu}.$$ 
We note that in this step we also need the elliptic estimates in a compact subset of $\Sigma$ to conclude that $\phi$ is bounded in any bounded neighbourhood of $\Sigma$. Since $\Sigma$ is $S$-dilation-nondegenerate, we conclude that $\phi=0$. 
\end{proof}
For $\delta\in\R$, we set
\begin{equation}
\begin{aligned}
Y_\delta&:=\{u\in L^2_\delta(\Sigma,|y|^{-N}):\,\langle u,v\rangle=0,\,\forall\,v\in L^2_{-\delta}(\Sigma,S)\},\\ 
X_\delta&:=Y_\delta\cap \mathcal{W}^{2,2}_\delta(\Sigma,|y|^{-N}),\\
Y^\bot_\delta&:=\{u\in L^2_\delta(\Sigma,|y|^{-N}):\,\langle u,v\rangle_\delta=0,\,\forall\,v\in Y_\delta\},\\
X^\bot_\delta&:=\{u\in \mathcal{W}^{2,2}_\delta(\Sigma,|y|^{-N}):\,\langle u,v\rangle_\delta=0,\,\forall\,v\in X_\delta\}=Y_\delta^\bot\cap \mathcal{W}^{2,2}_\delta(\Sigma,|y|^{-N}),
\end{aligned}
\end{equation}
so that any function $u\in \mathcal{W}^{2,2}_\delta(\Sigma)$ can be written as
$$u=u^\parallel+u^\bot\in X_\delta\oplus X^\bot_\delta.$$
The projections $u\in \mathcal{W}^{2,2}_{\delta}(\Sigma,|y|^{-N}) \mapsto u^{\parallel}\in X_{\delta}$ and $u\in \mathcal{W}^{2,2}_{\delta}(\Sigma,|y|^{-N}) \mapsto u^{\bot}\in X_{\delta}^{\bot}$ are continuous with respect to the $L^2_\delta(\Sigma)$-norm, since
\begin{equation}
\label{proj-cont}
\|u^\parallel\|^2_{L^2_\delta(\Sigma)}+\|u^\bot\|^2_{L^2_\delta(\Sigma)}=\|u\|^2_{L^2_\delta(\Sigma)},\qquad\forall\,u\in \mathcal{W}^{2,2}_\delta(\Sigma,|y|^{-N}).
\end{equation}
\begin{lemma}
In the above notations, we have 
\begin{enumerate}
\item[i.] $Y^\bot_\delta=L^2_\delta(\Sigma,S)$.
\item[ii.] $X^\bot_\delta\subset W^{2,2}_\delta(\Sigma,S)$.
\end{enumerate} 
\end{lemma}
\begin{proof}
{\rm i.} $L^2_\delta(\Sigma,S)\subset Y_\delta^\bot$. In fact, for $u\in L^2_\delta(\Sigma,S)$ and $v\in Y_\delta$, we have
$$\langle u,v\rangle_\delta=\langle u\Gamma^2_{-\delta},v\rangle=0,$$
since $u\Gamma^2_{-\delta}\in L^2_{-\delta}(\Sigma,S)$.\\

The opposite inclusion is proved as follows. First we note that, for any $v\in L^2_{-\delta}(\Sigma,S)$,
\begin{equation}\label{eqn:u_ucomrho}
\begin{aligned}
\int_\Sigma (u-u\circ\rho)v|y|^{-N}d\sigma&=
\int_\Sigma uv|y|^{-N}d\sigma-\int_\Sigma u(y')(v\circ\rho^{-1})(y')|\rho^{-1}(y')|^{-N}|\det(\rho)|d\sigma\\
&=\int_\Sigma uv|y|^{-N}d\sigma-\int_\Sigma u(v\circ\rho^{-1})|y|^{-N}d\sigma\\ &=0,
\end{aligned}
\end{equation}
since $v=v\circ\rho^{-1}$. Hence $u-u\circ\rho\in Y_\delta$.

Next, we show that $u\circ\rho\in Y^\bot_\delta$, for any $u\in Y^\bot_\delta$. In order to do so, we note that, if $v\in Y_\delta$, then $v\circ\rho\in Y_\delta$ for any $\rho\in S$. This is shown with a change of variables as above, which yields that for any $v\in Y_\delta$ we have
$$\int_\Sigma (u\circ\rho) v \Gamma^2_{-\delta}|y|^{-N}d\sigma=
\int_\Sigma u(v\circ\rho^{-1}) \Gamma^2_{-\delta}|y|^{-N}d\sigma=0.$$

Finally, from the previous comment and \eqref{eqn:u_ucomrho}, we conclude that for any $u\in Y^\bot_{\delta}$ and for any $\rho\in S$, we have $u=u\circ\rho$.\\

{\rm ii.} If $u\in X^\bot_\delta$, then $\langle u,v\rangle_\delta=0$ for any $v\in X_\delta$. Since $X_\delta$ is dense in $Y_\delta$, then the same orthogonality condition holds for any $v\in Y_\delta$, so that $X^\bot_\delta\subset Y^\bot_\delta=L^2_\delta(\Sigma,S)$. Since $X^\bot_\delta\subset \mathcal{W}^{2,2}_\delta(\Sigma,|y|^{-N})$, we have the statement. This concludes the proof of the lemma. 
\end{proof}

\begin{remark}
\label{rem-projections}
\begin{enumerate}
\item We note that $A_\delta(W^{2,2}_\delta(\Sigma,S))\subset L^2_\delta(\Sigma,S)$, which yields that $A_\delta(X_\delta)\subset Y_\delta$. In fact, for any $u\in X_\delta$ and $v\in W^{2,2}_{-\delta}(\Sigma,S)$, we have
$$\langle\mathcal{L}u,v\rangle=\langle u,\mathcal{L}v\rangle=0,$$
since $\mathcal{L}v\in L^2_{-\delta}(\Sigma,S)$. Using that $W^{2,2}_{-\delta}(\Sigma,S)$ is dense in $L^2_{-\delta}(\Sigma,S)$, we have the statement.
\item $A_\delta(X_\delta^\bot)\subset A_\delta(W^{2,2}_{\delta}(\Sigma,S))\subset L^2_{\delta}(\Sigma,S)=Y_\delta^\bot.$
\item For any $u\in K_\delta$, the projections onto $X_\delta$ and $X_\delta^\bot$ satisfy $u^\parallel\in K_\delta$ and $u^\bot\in K_\delta$. In fact $u^\parallel\in X_\delta$, $u^\bot\in X_\delta^\bot$ and by point $(2)$ we have $\mathcal{L}u^\parallel=-\mathcal{L}u^\bot\in Y_\delta\cap Y_\delta^\bot=\{0\}$. Equivalently $$K_\delta=(K_\delta\cap X_\delta)\oplus (K_\delta\cap X_\delta^\bot).$$
As a consequence, if the set $\{v_{\delta,j}\}_{1\le j\le d(\delta)}$ is an  orthonormal basis (with respect to the scalar product in \eqref{def-scalar-prod}) for  $K_{\delta}$ and $0<d'(\delta):=\dim(K_\delta\cap X_\delta)<d(\delta)$, up to a change of basis in $K_\delta$ we can assume that $v_{\delta,1},\dots,v_{\delta,d'(\delta)}\in K_\delta\cap X_\delta$ and $v_{\delta,d'(\delta)+1},\dots,v_{\delta,d(\delta)}\in K_\delta\cap X^\bot_\delta$.
\end{enumerate}
\end{remark}
Now we are ready to prove Theorem \ref{prop-right-inverse-Jacobi}.
\begin{proof}
Recall that $\bar{\delta} = \frac{N-2}{2} +\bar{\nu}$, so that $2-N-\bar{\nu}=-\frac{N-2}{2}-\bar{\delta}$.

Since $f\in C^{0,\alpha}_{\nu-2}(\Sigma,S)$ with $\nu>2-N-\bar{\nu}=-\frac{N-2}{2}-\bar{\delta}$, $\nu\notin\{-\frac{N-2}{2}\pm{\rm Re}(\Lambda_j)\}_{j\ge 0}$, then $g:=|y|^{\frac{N+2}{2}}f\in L^{2}_{\delta}(\Sigma,S)$ for any $\delta>\frac{N-2}{2}+\nu=:\delta'$.\\ 

Fix $\delta\in \R$ such that 
\begin{equation}\label{eqn:choice_delta_Teo1.2}
\delta'<\delta<\min\{\{\pm{\rm Re}(\Lambda_j)\}_{j\ge 0}\cap(\delta',\infty)\}.    
\end{equation}

By Proposition \ref{prop-right-inverse} we can find a solution $u\in \mathcal{W}^{2,2}_{\delta}(\Sigma,|y|^{-N})$ to $$\mathcal{L}u=g-\sum_{j=1}^{d(-\delta)}\mu_j v_{-\delta,j}\Gamma^2_{\delta},\qquad\text{$\mu_j:=\langle g,v_{-\delta,j}\rangle$ in $\Sigma$}$$ 
such that $$\|u\|_{W^{2,2}_{\delta}(\Sigma)}+\sum_{j=1}^{d(-\delta)}|\mu_j|\le c\|g\|_{L^2_{\delta}(\Sigma)}.$$
Using that $g\in L^2_{\delta}(\Sigma,S)$, we have $\langle g,v\rangle=0$ for any $v\in K_{-\delta}\cap X_{-\delta}$, thus by point $(3)$ of Remark \ref{rem-projections} $$\Gamma^2_{-\delta}(\mathcal{L}u-g)\in K_{-\delta}\cap X_{-\delta}^\bot\subset K_{-\delta}\cap W^{2,2}_{-\delta}(\Sigma,S)=K_{-\delta}(S).$$

Our choice of $\delta$ in \eqref{eqn:choice_delta_Teo1.2} implies that 
$$-\delta<-\frac{N-2}{2}-\nu\le\bar{\delta}.$$ 

Proposition \ref{prop-non-degeneracy} yields that $K_{-\delta}(S)=\{0\}$, so that $\mathcal{L}u=g$ and $\|u\|_{W^{2,2}_{\delta}(\Sigma)}\le c\|g\|_{L^2_{\delta}(\Sigma)}$. 

Using the decomposition $$u=u^\parallel+u^\bot\in X_{\delta}\oplus X^\bot_{\delta}$$
Remark \ref{rem-projections} and the symmetry properties of $\zeta_0$, it is possible to see that $$\mathcal{L}u^\parallel=g-\mathcal{L}u^\bot\in  Y_{\delta}\cap Y^\bot_{\delta}=\{0\}.$$
Therefore $u^\bot\in W^{2,2}_{\delta}(\Sigma,S)$ solves $\mathcal{L}u^\bot=g$ and $$\|u^\bot\|_{L^{2}_{\delta}(\Sigma)}\le\|u\|_{L^{2}_{\delta}(\Sigma)}\le c\|g\|_{L^2_{\delta}(\Sigma)}.$$
By Proposition $8.3.1$ of \cite{P}, we can see that $\|u^\bot\|_{W^{2,2}_{\delta}(\Sigma)}\le c\|g\|_{L^2_{\delta}(\Sigma)}$.\\


As a consequence, equation (\ref{eq-Jacobi}) is solved by setting $\phi:=|y|^{-\frac{N-2}{2}}u^\bot$.\\

In order to prove (\ref{est-Jacobi-eq-lin}) we observe that, since $(\delta',\delta)\cap\{\pm{\rm Re}(\Lambda_j)\}_{j\ge 0}=\emptyset$, by Proposition $12.2.1$ of \cite{P}, if $f\in C^{0,\alpha}_{\nu-2}(\Sigma)$, then $u^\bot\in C^{2,\alpha}_{\delta'}(\Sigma)$ and
$$\|u^\bot\|_{C^{2,\alpha}_{\delta'}(\Sigma)}\le c(\||y|^\frac{N+2}{2}f\|_{C^{0,\alpha}_{\delta'}(\Sigma)}+\|u^\bot\|_{L^2_{\delta}(\Sigma)}).$$
Finally we note that
\begin{equation}\notag
\begin{aligned}
    &\||y|^\frac{N+2}{2}f\|_{C^{0,\alpha}_{\delta'}(\Sigma)}\le c\|f\|_{C^{0,\alpha}_{\nu-2}(\Sigma)},\\
    &\|u^\bot\|_{L^2_{\delta}(\Sigma)}\le c\||y|^\frac{N+2}{2}f\|_{L^2_{\delta}(\Sigma)}\le c\||y|^\frac{N+2}{2}f\|_{C^{0,\alpha}_{\delta'}(\Sigma)}\le c\|f\|_{C^{0,\alpha}_{\nu-2}(\Sigma)},
    \end{aligned}
\end{equation}
which concludes the proof of the estimate, since $\phi=|y|^{-\frac{N-2}{2}}u^\bot$.

\end{proof}   
\section{Examples}
In this section we study the nondegeneracy properties of some specific minimal hypersurfaces, including the ones constructed in Theorem \ref{th-Sigma}, \ref{th-Al} and \ref{th-Sigma-low-dim}. Then, we consider the Jacobi equation $J_{\Sigma}\phi =f$ on the hypersurfaces $\Sigma$ constructed in Theorems \ref{th-Sigma} and \ref{th-Al} with $f={\rm tr}(A_\Sigma^3)$ and solve it by applying the estimate \eqref{est-Jacobi-eq-lin} from Theorem \ref{prop-right-inverse-Jacobi}. In particular, we are interested in the decay rate of the solution $\phi$ at infinity, which plays a crucial in role in the construction of solutions to the Allen-Cahn equation (see for instance \cite{agudelo2022doubling,agudelo2022k,del2011giorgi,del2013entire,pacard2013stable}). Finally, we will show an example of degenerate hypersurface.

\color{black}\subsection{Examples of $S$-dilation-nondegenerate minimal hypersurfaces}
The first example is given by stable minimal hypersurfaces which does not intersect their asymptotic cone.
\begin{proposition}
\label{prop-dil-nondeg}
Assume (H1)-(H4). Assume also that $\Sigma$ is stable and that $\Sigma$ does not intersect the cone $C$. Then, $\Sigma$ is dilation-nondegenerate.

If in addition $\Sigma$ is $S$-invariant, then $\Sigma$ is also $S$-dilation-nondegenerate.
\end{proposition}
\begin{proof} Recall the space $D_\nu(\Sigma)$ in \eqref{def:jacobi_fields_decay} and the Definition \ref{def:S_dila_nond}. We show first that $\bar{\nu}>2-N$ and then we show that $D_\nu(\Sigma)=\{0\}$ for any $\nu<\bar{\nu}$. We proceed as follows. Since $\Sigma$ is stable, from Remark $3.2$ in \cite{agudelo2024jacobi}, we conclude that
$$\Lambda_0^2=\Big(\frac{N-2}{2}\Big)^2 + \lambda_0\ge 0 \quad \hbox{and}\quad \bar{\nu}\in\left\{-\frac{N-2}{2}\pm\Lambda_0\right\}.
$$ 

Using for instance Lemma 2.1 in \cite{agudelo2024jacobi}, $\lambda_0\le 0$ and equality holds if and only if $\Gamma$ is a union of equatorial spheres. These observations yield that 
\begin{equation}
\label{barnu>2-N}
\bar{\nu}\ge-\frac{N-2}{2}-\Lambda_0\ge 2-N.
\end{equation}

Now, to show that $\bar{\nu}>2-N$ we analyze two separate cases: $\lambda_0<0$ and $\lambda_0=0$.

\vskip 3pt
If $\lambda_0<0$ the strict inequality $\bar{\nu}>2-N$ follows from \eqref{barnu>2-N}. If $\lambda_0=0$, then for some $k\in \mathbb{N}\cup \{\infty\}$, $\Gamma$ is the union of $k$ equatorial spheres in $\mathbb{S}^N$. If $k=1$, then $C=\R^{N}$ is strictly minimising in $\R^{N+1}$, since $N\ge 3$ (see \cite{pacard2013stable, HS}), which yields that $$\bar{\nu}=-\frac{N-2}{2}+\Lambda_0=0>2-N.
$$

The case $k>1$ does not occur. If it did, then $C$ would be a cylinder, that is $C=\tilde{C}\times \R$, for some cone $C\subset\R^{N}$. Let us assume for the moment, after performing rotations if necessary, that $\Gamma$ contains an equatorial sphere in the hyperplane $\{x_1=0\}$ of $\R^{N+1}$ and another one in the hperplane $\{x_2=0\}$. Then the singular set of $C$ would contain a line and this would contradict the existence of asymptotic hypersurfaces in the sense of our definition. This concludes the proof of $\bar{\nu}> 2-N$.

\vskip 3pt

Now, let $\nu <\bar{\nu}$. We show that $D_{\nu}(\Sigma)=\{0\}$. To do so, we proceed as follows. First, write $\nu = -\frac{N-2}{2} +\delta$ and notice that $\delta < \bar{\delta}:=\frac{N-2}{2}-\bar{\nu}$. In view of Proposition $3.1$ of \cite{agudelo2024jacobi}, $A_{\delta}$ is injective. The conclusion follows by noticing that
$$
D_{\nu}(\Sigma) \subset |y|^{-\frac{N-2}{2}}K_{\delta}=\{0\}.
$$
\end{proof}
Considering our explicit examples of minimal hypersurfaces asymptotic to a cone, we have the following result.
\begin{lemma}
\label{lemma-nondeg-examples}
\begin{enumerate}
\item The minimal hypersurfaces constructed in Theorem \ref{th-Sigma}  are dilation-nondegenerate and fulfill $\Lambda_0>0$ and $\bar{\nu}=-\frac{N-2}{2}+\Lambda_0$ if the underlying cone $C$ is strictly minimising, $\bar{\nu}=-\frac{N-2}{2}-\Lambda_0$ if $C$ is minimising but not strictly.
\item The minimal hypersurfaces constructed in Theorem \ref{th-Al}  are dilation-nondegenerate and fulfill $\Lambda_0>0$ and $\bar{\nu}=-\frac{N-2}{2}+\Lambda_0$.
\item The minimal hypersurfaces constructed in Theorem \ref{th-Sigma-low-dim} are $O(m)\times O(n)$-nondegenerate and fulfill ${\rm Re}(\Lambda_0)=0$, $\bar{\nu}=-\frac{N-2}{2}=-\frac{N-2}{2}\pm{\rm Re}(\Lambda_0)$.
\end{enumerate}
\end{lemma}
\begin{proof}
\begin{enumerate}
\item The dilation-nondegeneracy follows from Proposition \ref{prop-dil-nondeg}, since such hypersurfaces are stable and do not intersect their asymptotic cone $C$. The statement about $\bar{\nu}$ follows from Theorem $3.2$ of \cite{HS}.
\item The dilation-nondegeneracy follows once again from Proposition \ref{prop-dil-nondeg}, exactly as above. For the computation of $\bar{\nu}$ we refer to Remark $5.2$ of \cite{agudelo2024jacobi}.
\item Regarding the hypersurfaces constructed in Theorem \ref{th-Sigma-low-dim}, the result follows from the fact that $\Lambda_0^2=\left(\frac{N-2}{2}\right)^2-(N-1)<0$, since $3\le N\le 6$, and Theorem $5.2$ of \cite{agudelo2024jacobi}.
\end{enumerate}
\end{proof}
\begin{remark}
Note that Lemma \ref{lemma-nondeg-examples} implies Theorem \ref{theo:examples}.
\end{remark}
\color{black}Consider again a cone $C$ and a hypersurface $\Sigma$ satisfying (H1)-(H4) and such that $\Sigma$ does not intersect $C$. It is natural to ask about the properties of the Jacobi fields of $\Sigma$ when we allow them to decay more slowly (or possibly explode) at infinity. We will see (see Proposition \ref{prop-jac-dil} below) that, for $\nu\in(\bar{\nu},-\frac{N-2}{2}+\Lambda_1)$, the Jacobi fields in the space $D_\nu(\Sigma)$ are geometric. More precisely, $D_\nu(\Sigma)$ is either trivial or generated by the Jacobi field $\zeta_0$ corresponding to dilations.\\

For this purpose, we fix $R>0$ (large enough) and look at the equation $$J_\Sigma\phi=0\quad\text{in }\Sigma_R.$$ 
Introducing the change of variables $\phi(y)=|y|^{-\frac{N-2}{2}}v(y)$ (see \eqref{eq-SigmaR}), the equation for $v$ reads 
$$\mathcal{L}v=0\qquad\text{in }\Sigma_R.$$
We write the points $y\in\Sigma_R$ as 
\begin{equation}\label{eq:coord_asympt_cone}
y=r\theta+w(r,\theta)\nu_\Sigma(\theta),\qquad r>R.    
\end{equation}

With an Emden-Fowler change of variables $r=e^t$ and $\tilde{v}(t,\theta)=v(y)$, the equation for $\tilde{v}(t,\theta)$ reads
\begin{equation}\label{eq:Jac_Emd_Fowl_remainder}
\partial_{tt}\tilde{v}+\Delta_\Gamma\tilde{v}+\left(|A_\Gamma|^2-\left(\frac{N-2}{2}\right)^2\right)\tilde{v}+\mathcal{R}(\tilde{v})=0\qquad\text{in }(\log R,\infty)\times\Gamma,
\end{equation}
where $\mathcal{R}(\tilde{v})$ is defined in \eqref{eq-SigmaR} and satisfies \eqref{est-remainder}. Recalling that $R>0$ is large, then $t_0:=\log(R)$ is large as well and hence a function $\tilde{v}\in C^2((t_0,\infty)\times \Gamma)$ solving the equation
\begin{equation}
\label{eq-approx-SigmaR}
\partial_{tt}\tilde{v}+\Delta_\Gamma\tilde{v}+\left(|A_\Gamma|^2-\left(\frac{N-2}{2}\right)^2\right)\tilde{v}=0\qquad\text{in }(\log R,\infty)\times\Gamma
\end{equation}
is a {\it good approximation} of a solution of the equation \eqref{eq:Jac_Emd_Fowl_remainder}.

\medskip
Next, for $j\in \mathbb{N}$, let  $E_j$ denote the eigenspace of $-J_\Gamma$ relative to $\lambda_j$ (see \eqref{ineq:eigval_Gamma}). Then, for any $j\ge 0$ and $\phi\in E_j$, the function $U^\pm_{j,\phi}(t,\theta):=e^{\pm\Lambda_jt}\phi(\theta)$ is a smooth solution to (\ref{eq-approx-SigmaR}), where $\Lambda_j$ is defined in \eqref{def:discrim}.

\vskip 3pt
Also, for $j\in \mathbb{N}$ and for $y\in \Sigma_R$, set $u^{\pm}_{j,\phi}(y):=U_{j,\phi}^{\pm}(t,\theta)$, where $y$ is written as in \eqref{eq:coord_asympt_cone} and $r=e^t$ for $t>t_0$.

\begin{lemma}
Let $\eta>0$ be as in \eqref{est-remainder}. For $R:=e^{t_0}$ large enough and fixed, and for any $j\ge 0$ and any $\phi\in E_j$, there exists two solutions $w^\pm_{j,\phi}$ to
$$\mathcal{L}w^\pm_{j,\phi}=0\qquad\text{in }\Sigma_R$$
such that for any $\delta<\pm{\rm Re}(\Lambda_j)-\eta$, $w^\pm_{j,\phi}-u^\pm_{j,\phi}\in \Gamma_\delta L^2(\Sigma_R,|y|^{-N})$. Moreover, the mapping $$\phi\mapsto w^\pm_{j,\phi}$$
is linear.
\end{lemma}

\begin{proof} Enough to use \eqref{est-remainder} and argue exactly as in Lemma $11.1.3$ of \cite{P}.
\end{proof}

For $j\geq 0$ and $\phi \in E_j$, write $W_{j,\phi}^{\pm}(t,\theta):=w^{\pm}_{j,\phi}(y)$, where $y$ is written as in \eqref{eq:coord_asympt_cone} and $r=e^t$ for $t>t_0$.

\vskip 3pt
Next, let $\chi:\R\to\R$ be a smooth cutoff function such that $\chi=0$ in $(-\infty,\log(R+1))$ and $\chi=1$ in $(\log(R+2),\infty)$. Following \cite{P}, for $\delta<0$ we introduce the \textit{Deficiency space}
\begin{equation}
\mathcal{D}_\delta:={\rm span}\big\{\chi (t) W^\pm_{j,\phi}(t,\theta):\,\phi\in E_j\quad \hbox{and}\quad{\rm Re}(\Lambda_j)<-\delta, \quad j\ge 0 \big\}.
\end{equation}
{\begin{proposition}\label{prop-jac-dil}
Under the assumptions of Proposition \ref{prop-dil-nondeg}, we have $$D_\nu(\Sigma)\subset{\rm span}\{\zeta_0\}\qquad\forall\,\nu<-\frac{N-2}{2}+\Lambda_1.$$
\end{proposition}
\begin{proof}
Set $\nu:=-\frac{N-2}{2}+\delta$. Using that $K_{\delta'}\subset K_\delta$ if $\delta'<\delta$, it is enough to prove the statement for $\delta\in(\Lambda_0,\Lambda_1)$.\\

Arguing as in Corollary $11.3.1$ of \cite{P} and using the fact that the first eigenvalue $\lambda_0$ of 
$-J_\Sigma$ is simple, for $\delta\in(\Lambda_0,\Lambda_1)$ we have
$$\dim K_{\delta}=\frac{1}{2}\dim \mathcal{D}_{-\delta}=1.
$$
Since $v:=|y|^{\frac{N-2}{2}}\zeta_0\in K_{\delta}$ and $\delta\in(\Lambda_0,\Lambda_1)$, then $K_{\delta}={\rm span}\{|y|^{\frac{N-2}{2}}\zeta_0\}$, which yields that
$$
D_\nu(\Sigma)\subset|y|^{-\frac{N-2}{2}}K_{\delta}={\rm span}\{\zeta_0\},
$$
which concludes the proof.
\end{proof}
\begin{remark}
Proposition \ref{prop-jac-dil} shows that, for $\nu<-\frac{N-2}{2}+\Lambda_1$, the Jacobi fields in $D_\nu(\Sigma)$ are geometric, or in other words $D_\nu(\Sigma)\subset G(\Sigma)$.
\end{remark}}


\subsection{The Jacobi equation with an explicit right-hand side}

From now on, $\Sigma$ will be one of the minimal hypersurfaces constructed either in Theorem \ref{th-Al} or Theorem \ref{th-Sigma-low-dim}. In this part we consider the $S=O(m)\times O(n)$ as the group of symmetries and we stress that all these hypersurfaces are thus $S$-invariant.\\

Set $\Gamma=\sqrt{\frac{m-1}{N-1}}\mS^{m-1}\times \sqrt{\frac{n-1}{N-1}}\mS^{n-1}$, for $m,\,n\ge 2$, so that the eigenvalues of $-J_\Gamma$ satisfy
$$
-(N-1)=\lambda_0<0=\lambda_1\le\lambda_2\le\cdots.$$

Assume first that $m+n=N+1\ge 8$ and let $\Sigma=\Sigma^\pm_{m,n}$ be one of the hypersurfaces constructed in Theorem \ref{th-Al}. Then,
$$\bar{\delta}={\rm Re}(\Lambda_0)=\Lambda_0=\sqrt{\left(\frac{N-2}{2}\right)^2-(N-1)}>0.
$$

Thus, $\bar{\nu}:=-\frac{N-2}{2}+\Lambda_0$ and 
$$
-\frac{N-2}{2}-\Lambda_0<\bar{\nu}<-1<-\frac{N-2}{2}+\Lambda_1=0.
$$

In particular, $-1 \in \big(\bar{\nu},-\frac{N-2}{2}+\Lambda_1\big)$ is not an indicial root and since ${\rm tr} A_\Sigma^3\in C^{0,\alpha}_{\nu}(\Sigma,S)$ for any $\nu\ge -3$, equation \eqref{eq-trA3} is solvable in $C^{2,\alpha}_{-1}(\Sigma,S)$, thanks to Theorem \ref{prop-right-inverse-Jacobi}. Thus, we recover the result of Proposition 2.2 in  \cite{agudelo2022doubling}.\\

On the other hand, if $4\le N+1\le 7$ and $\Sigma=\Sigma_{m,n}$ is one of the hypersurfaces constructed in Thoerem \ref{th-Sigma-low-dim}, we have $\bar{\delta}={\rm Re}(\Lambda_0)=0$ and $\Lambda_1=\frac{N-2}{2}$. Hence, for $N=5,6$ we have
$$-\frac{N-2}{2}\pm{\rm Re}(\Lambda_0)=-\frac{N-2}{2}<-1<-\frac{N-2}{2}+\Lambda_1=0,$$
and proceeding as above, we find that equation \eqref{eq-trA3} is solvable in $C^{0,\alpha}_{-1}(\Sigma,S)$ thanks to Theorem \ref{prop-right-inverse-Jacobi}.\\

For $N=4$, we apply Theorem  \ref{prop-right-inverse-Jacobi} for $\nu$ satisfying
$$-\frac{N-2}{2}\pm{\rm Re}(\Lambda_0)=-\frac{N-2}{2}=-1<\nu<-\frac{N-2}{2}+\Lambda_1=0,$$
but not for $\nu=-1$, since it corresponds to the indicial root ${\rm Re}(\Lambda_0)=0$. This means that equation \eqref{eq-trA3} is still solvable, but the solution decays slightly slower than $|y|^{-1}$ at infinity. More precisely, for any $\nu\in(-1,0)$, we can find a solution $\phi\in C^{2,\alpha}_\nu(\Sigma,S)$.\\

For $N=3$, we apply Theorem \ref{prop-right-inverse-Jacobi} for $\nu$ satisfying
$$-\frac{N-2}{2}\pm{\rm Re}(\Lambda_0)=-\frac{N-2}{2}=-\frac{1}{2}<\nu<-\frac{N-2}{2}+\Lambda_1=0,$$
This means that we can find a solution to equation \eqref{eq-trA3} whose decay at infinity is slightly slower than $|y|^{-\frac{1}{2}}$. More precisely, for any $\nu\in(-\frac{1}{2},0)$, we can find a solution $\phi\in C^{2,\alpha}_\nu(\Sigma,S)$.

\subsection{An example of dilation degenerate hypersurface}

In order to construct a dilation degenerate minimal hypersurface, we prove the following Proposition.
\begin{proposition}\label{prop-Plateau}
For any $R>0$, there exists $\alpha=\alpha(R)>0$ such that the Plateau problem
\begin{equation}\label{Plateau-pb}
\begin{aligned}
{\rm div}\left(\frac{\nabla u}{\sqrt{1+|\nabla u|^2}}\right)&=0\qquad\text{in}\,\R^N\setminus\overline{B}_R  \\
u&=\alpha(R) \qquad\text{on}\,\partial B_R
\end{aligned}
\end{equation}
has a positive radial nonconstant solution $u_R\in C^\infty(\R^N\setminus\overline{B}_R)$ such that. More precisely, there exists $v_R\in C^\infty(R,\infty)$ such that $u_R(x)=v_R(|x|)>0$ and
\begin{equation}\label{est-sol-plateau}
v_R(r)=cr^{2-N}(1+o(1)),\qquad v_R'(r)=-c(N-2)r^{1-N}(1+o(1))\qquad\text{as}\,r\to\infty
\end{equation}
for some constant $c>0$.
\end{proposition}
\begin{proof}
If we look for a radial solution $u(x)=v(|x|)$ to \eqref{Plateau-pb}, $v$ has to satisfy equation 
\begin{equation}\label{Plateau-pb-v}
\frac{1}{r^{N-1}}\left(r^{N-1}\frac{v'}{\sqrt{1+(v')^2}}\right)'=0\qquad\forall\,r>R,
\end{equation}
or equivalently 
\begin{equation}
\label{eq-v'}
r^{N-1}\frac{v'}{\sqrt{1+(v')^2}}=\kappa\qquad\forall\,r>R,    
\end{equation}
for some constant $\kappa\in\R$. If $\kappa=0$ we have the constant solution, hence we are interested in the case $\kappa\ne 0$. Since we are looking for a positive solution, we take $\kappa=-c<0$. As a consequence \eqref{eq-v'} becomes 
$$\frac{v'}{\sqrt{1+(v')^2}}=-\frac{c}{r^{N-1}}\qquad\forall\,r>R,$$
or equivalently
\begin{equation}
v'(r)=-\frac{cr^{1-N}}{\sqrt{1-((cr)^{1-N})^2}}\qquad\forall r>R   
\label{eq-v'-inv}
\end{equation}
provided $c$ is such that $cR^{1-N}=1$, that is $c=R^{N-1}$. As a consequence, by the fundamental Theorem of calculus, the function
$$v(r):=\int_r^{\infty}\frac{\rho^{1-N}}{\sqrt{R^{2-2N}-\rho^{2-2N}}} \,d\rho>0,\qquad\forall\,r>R$$
is smooth in $(R,\infty)$ and fulfills \eqref{Plateau-pb-v}. Hence the smooth function $u(x):=v(|x|)$ is a solution to the Plateau problem \eqref{Plateau-pb} provided we set $$\alpha(R):=\int_R^{\infty}\frac{\rho^{1-N}}{\sqrt{R^{2-2N}-\rho^{2-2N}}} \,d\rho\in(0,\infty).$$
The asymptotic behaviour of $v'$ as $r\to\infty$ is directly given by \eqref{eq-v'-inv} and the one of $v$ follows from the de l'Hopital Theorem.
\end{proof}
\begin{remark}
From the proof of Proposition \ref{prop-Plateau}, it follows that the solution to the Plateau problem constructed in Proposition \ref{prop-Plateau} fulfills $v'(r)<0$ for any $r>R$ and $v'(r)\to-\infty$ as $r\to R^+$.
\end{remark}
\begin{proof}
Let us consider, for any $R>0$, the minimal graph constructed in Proposition \ref{prop-Plateau}, that is the hypersurface $$\Sigma_R:=\{(x,u_R(x)):\,x\in\R^N\setminus\overline{B}_R\}.$$
Note that $\Sigma_R$ is asymptotic to $\R^{N}\times\{0\}$ in the sense of (H3) with $K_R=\emptyset$.\\

The Jacobi field coming from dilations satisfies
\begin{equation}
\begin{aligned}
\zeta_0(y)&=\langle (x,u_R(x)),\frac{(-\nabla u_R(x),1)}{\sqrt{1+|\nabla u_R(x)|^2}}\rangle=\frac{-rv_R'(r)+v_R(r)}{\sqrt{1+(v'_R(r))^2}}\\
&=c(N-2)r^{2-N}(1+o(1))+cr^{2-N}(1+o(1))=c(N-1)r^{2-N}(1+o(1))\qquad r\to\infty.  
\end{aligned}    
\end{equation}
This shows that $\bar{\nu}=2-N$ if $N\ge 3$.
\end{proof}
\begin{remark}
Let $\Omega\subset\R^N$ be an open set. A function $u\in C^2(\Omega)$ represents a minimal graph if and only if
$${\rm div}\left(\frac{\nabla u}{\sqrt{1+|\nabla u|^2}}\right)=0\qquad\text{in}\,\Omega,$$
or equivalently
\begin{equation}\label{min-graph-equiv}
\Delta u=\frac{\nabla^2 u[\nabla u,\nabla u]}{1+|\nabla u|^2}\qquad\text{in}\,\Omega.   
\end{equation}
The solution $u_R(x)=v_R(|x|)$ constructed in Proposition \ref{prop-Plateau} satisfies $v''_R(r)>0$ for any $r>R$, hence, due to \eqref{min-graph-equiv}, $u_R>0$ is subharmonic in $\R^N\setminus\overline{B}_R$. Moreover, it satisfies $\lim_{|x|\to\infty}u_R(x)=0$. As a consequence, by Lemma $6.3$ of \cite{farina2016symmetry}, there exists a constant $C>0$ such that $$0<u_R(x)\le C |x|^{2-N}\qquad\forall\,x\in\R^N\setminus\overline{B}_R.$$
Therefore, the decay rate given by \eqref{est-sol-plateau} is sharp.
\end{remark}
\section{An explicit example with symmetry}

In this part we prove the following result.
\begin{theorem}
\label{th-sym-sharp}
Let $\Sigma\subset\R^{N+1}$ be one of the minimal hypersurfaces constructed in Theorem \ref{th-Al} or Theorem \ref{th-Sigma-low-dim}. Then there exists a solution $\phi$ to equation \eqref{eq-trA3} such that for any $y\in \Sigma$,
\begin{equation}
\label{est-phi-sharp}
\begin{aligned}
&|\phi(y)|(|y|+1)\le c\|g\|_{C^{0,\beta}_3(\Sigma)},\qquad &N\ge 5\\
&|\phi(y)|\frac{|y|+1}{\log(|y|+2)}\le c\|g\|_{C^{0,\beta}_3(\Sigma)},\qquad &N=4\\
&|\phi(y)|\frac{(|y|+1)^{1/2}}{\log(|y|+2)}\le c\|g\|_{C^{0,\beta}_3(\Sigma)},\qquad &N=3,
\end{aligned}
\end{equation}
where $g:={\rm tr}(A_\Sigma^3)$.
\end{theorem}
These solutions can be obtained by Theorem \ref{prop-right-inverse-Jacobi} too. However, Theorem \ref{th-sym-sharp} gives the {\it sharp decay} at infinity of such solutions, using the symmetry of these hypersurfaces.\\ 

The remaining part of the section is dedicated to the proof of Theorem \ref{th-sym-sharp}, which consists various steps.\\

Being $\Sigma$ is $O(m)\times O(n)$ invariant, so it can be parametrised by
$$(s,{\tt x},{\tt y})\in(0,\infty)\times S^{m-1}\times S^{n-1}\mapsto\Psi(s,{\tt x},{\tt y}):=(a(s){\tt x}, b(s){\tt y})\in\Sigma.$$
In other words, $\Sigma$ can be seen as a curve $\gamma(s):=(a(s),b(s))$ in the quadrant $$Q:=\{(a,b)\in\R^2:\,a\ge 0,\,b\ge 0\}.$$ 
Without loss of generality, we can assume that $\gamma$ is parametrised by arc-length, in the sense that $(a')^2+(b')^2=1$. In the variables $(s,\x,\y)$, the mean curvature equation and the second fundamental form are given by
\begin{equation}
\label{0_mc_O(m)O(n)_invariant}
H_{\Sigma}=-a'' b'+b''a'+(m-1)\frac{b'}{a}-(n-1)\frac{a'}{b}=0
\end{equation}
and
\begin{equation}
\label{1_mc_O(m)O(n)_invariant}
\begin{aligned}
|A_{\Sigma}|^2 &= \left(-a''b'+ a'b''\right)^2 + (m-1)\left(\frac{b'}{a} \right)^2 + (n-1)\left(\frac{a'}{b} \right)^2.
\end{aligned}
\end{equation}

Observe that in coordinates $(s,\x,\y)$, the expressions \eqref{0_mc_O(m)O(n)_invariant} and \eqref{1_mc_O(m)O(n)_invariant} depend only on $s$ due to the $O(m)\times O(n)$ invariance of $\Sigma$. For a more detailed description of $\Sigma$, see \cite{agudelo2022doubling}.

Denote $$\alpha=(m-1)\frac{a'}{a}+(n-1)\frac{b'}{b}, \qquad \beta= \left(-a''b'+ a'b''\right)^2 + (m-1)\left(\frac{b'}{a} \right)^2 + (n-1)\left(\frac{a'}{b} \right)^2.
$$

Since the right-hand side $g={\rm tr}(A_\Sigma^3)$ of (\ref{eq-trA3}) is $O(m)\times O(n)$ invariant, we look for solutons $\psi$ of the Jacobi equation, which are $O(m)\times O(n)-$invariant, so that the equation reduces to an ODE in the $s$-variable, namely
\begin{equation}\notag
\psi_{ss}+\alpha\psi_s+\beta\psi=f,
\end{equation}
where $f(s):=g(y)$, $\psi(s):=\phi(y)$. Introducing the Emden-Fowler change of variables $s=e^t>0$ and writing $\psi(s)=p(t)u(t)$, where
$$p(t)=\exp\left(-\int_0^t \frac{\alpha(e^\tau)e^\tau-1}{2}d\tau\right),$$
we reduce ourselves to solve the equation
\begin{equation}
\label{Jacobi-eq-EF}
u_{tt}+V(t)u=\tilde{f}\qquad\text{in $\R$,}
\end{equation}
where (see for instance Section 2 in \cite{agudelo2022doubling} and Section 5 in \cite{agudelo2024jacobi})
$$\tilde{f}(t)=\frac{e^{2t}}{p(t)}f(e^t),\qquad V(t)=-\frac{1}{4}(\alpha(e^t)e^t-1)^2+\frac{1}{2}(\alpha'(e^t)e^{2t}+\alpha(e^t)e^t)+\beta(e^t)e^{2t}$$
and 
$$
p(t) =
\left\{
\begin{aligned}
e^{-\frac{n-2}{2}t}\big(
1 + O(e^{2t})\big)
,\quad  &t < T_0\\
e^{-\frac{N-2}{
2}t} \big(
1 + O(e^{-t})
\big)
,\quad  &t > T_1.
\end{aligned}
\right.
$$

From the results in Section 5 in \cite{agudelo2024jacobi}, following the same ideas from Section 2 in \cite{agudelo2022doubling} and from Theorem \ref{th-Al} or Theorem \ref{th-Sigma-low-dim}, we can conclude that the asymptotic behavior of $V(t)$ at infinity is still given by $(2.45)$ of \cite{agudelo2022doubling}, that is
\begin{equation}
\label{as-V}
V(t)=
\begin{cases}
-\frac{(n-2)^2}{4}+O(e^{\gamma t}) \qquad\text{as $t\to-\infty$}\\
-\frac{(N-2)^2}{4}+N-1+O(e^{-\gamma t})\qquad\text{as $t\to\infty$}
\end{cases}
\end{equation} 
for some $\gamma>0$. For the sake of completeness, we include the proof of  \eqref{as-V}.

In order to do so, we use Propositions $3$ and $5$ of \cite{Mazet_2017} and the fact that we can choose a basis of the $O(m)\times O(n)-$invariant subspaces of the Kernel of $J_{C_{m,n}}$ consisting of $2$ elements $\zeta_\pm$ such that $\zeta_\pm(s)=O(s^{-\frac{N-2}{2}})$ as $s\to\infty$. From these facts we observe that, outside a ball, $\Sigma$ is a normal graph of a function $\psi=\zeta+\varphi$ over the Lawson cone $C_{m,n}$, where $\zeta$ is in the $O(m)\times O(n)-$invariant Kernel of the Jacobi operator $J_{C_{m,n}}$ of the cone and $$|\varphi(s)|\le c s^{-(N-2)+2\delta}\qquad\forall s\ge s_0,$$
for some $\delta>0$.  In other words, the surface can be represented in the $(a,b)$-plane by a curve of the form
\begin{equation}\notag
(a,b)=\left(\sqrt{\frac{m-1}{N-1}},\sqrt{\frac{n-1}{N-1}}\right)s+\psi(s)\left(\sqrt{\frac{n-1}{N-1}},-\sqrt{\frac{m-1}{N-1}}\right).
\end{equation}
The zero mean curvature equation gives
\begin{equation}
\label{H=0-as-beh}
\begin{aligned}
-a''b'+a'b''&=(n-1)\frac{a'}{b}-(m-1)\frac{b'}{a}=\\
&=(n-1)\left(\sqrt{\frac{m-1}{n-1}}\frac{1}{s}+O(s^{-\frac{N+2}{2}})\right)
-(m-1)\left(\sqrt{\frac{n-1}{m-1}}\frac{1}{s}+O(s^{-\frac{N+2}{2}})\right)=\\
&=O(s^{-\frac{N+2}{2}}).
\end{aligned}
\end{equation}
Differentiating the arc-length relation $(a')^2+(b')^2=1$ and we get
\begin{equation}
\label{arc-lenght-diff}
a'' a'+b'' b'=0.
\end{equation}
Using that $a'\to\sqrt{\frac{m-1}{N-1}}$ and $b'\to\sqrt{\frac{n-1}{N-1}}$ as $s\to\infty$ and inverting the matrix associated to the system given by equations (\ref{H=0-as-beh}) and (\ref{arc-lenght-diff}), it is possible to see that $$a''=O(s^{-\frac{N+2}{2}}),\,b''=O(s^{-\frac{N+2}{2}})\qquad s\to\infty,$$
so that 
$$\frac{a''}{a}=O(s^{-\frac{N}{2}-2}),\,\frac{b''}{b}=O(s^{-\frac{N}{2}-2})\qquad s\to\infty.$$
As a consequence, since $3\le N\le 6$ we have
\begin{equation}
\begin{aligned}
\alpha'&=(m-1)\left(\frac{a''}{a}-\left(\frac{a'}{a}\right)^2\right)+(n-1)\left(\frac{b''}{b}-\left(\frac{b'}{b}\right)^2\right)\\
&=-\frac{N-1}{s^2}+O(s^{-\frac{N}{2}-2})\qquad s\to\infty.
\end{aligned}
\end{equation}
Finally, using that $\alpha(s)=\frac{N-1}{s}+O(s^{-1+\delta})$ and that we are setting $s=e^t$ for $t\in \R$, we have (\ref{as-V}).\\

In order to solve the Jacobi equation (\ref{Jacobi-eq-EF}) we divide $\R$ into $3$ intervals for the variable $t$, where we will analyze the behavior of the $O(m)\times O(n)-$invariant Jacobi fields.

\subsection{The interval $(-\infty, t_0)$}\label{subs-left}

We first consider an interval $(-\infty,t_0)$, with $t_0$ such that the Jacobi field $\zeta_0(y):=y\cdotp \nu_\Sigma(y)$ does not change sign for $t=\log(s)\in(-\infty,t_0)$. It is possible to see that, setting  $$u_+(t):=\frac{\zeta_0(y)}{p(t)}=\frac{a(e^t)b'(e^t)-a'(e^t)b(e^t)}{p(t)},\qquad s=e^t,$$
we have 
\begin{equation}
\label{Jacobi-field+as-beh-n>2}
0<c e^{\lambda t}\le u_+(t)\le Ce^{\lambda t},\,|\partial_t u_+(t)|\le Ce^{\lambda t}\qquad\forall\, t\in(-\infty,t_0),\,n\ge 3
\end{equation}
and
\begin{equation}
\label{Jacobi-field+as-beh-n=2}
0<c\le u_+(t)\le C,\,|\partial_t u_+(t)|\le C\qquad\forall\, t\in(-\infty,t_0),\,n=2.
\end{equation}
Therefore, using the variation of parameters formula, the other linearly independent Jacobi field in the interval $(-\infty,t_0)$ is defined by
$$u_-(t):=u_+(t)\int_{-\infty}^t \frac{1}{u_+(\tau)^2}d\tau$$
and its asymptotic behaviour is given by
\begin{equation}\notag
\label{Jacobi-field-as-beh-n>2}
0<ce^{-\lambda t}\le u_-(t)\le Ce^{-\lambda t},\,|\partial_t u_-(t)|\le Ce^{-\lambda t}\qquad\forall\, t\in(-\infty,t_0),\,n\ge 3
\end{equation}
and
\begin{equation}\notag
\label{Jacobi-field-as-beh-n=2}
0<-ct\le u_-(t)\le -Ct,\,|\partial_t u_-(t)|\le C|t|\qquad\forall\, t\in(-\infty,t_0),\,n=2.
\end{equation}
Multiplying by a constant if necessary, we can assume that the corresponding Wronskian determinant is $1$. Since $f$ is bounded in $(-\infty,t_0)$, we have $$|\tilde{f}|\le ce^{\frac{n+2}{2}t}\|g\|_{C^{1,\beta}_3(\Sigma)}\qquad\forall t\in(-\infty, t_0),$$
so that the solution
$$u(t):=u_+(t)\int_{-\infty}^t u_-(\tau)\tilde{f}(\tau)d\tau-u_-(t)\int_{-\infty}^t u_+(\tau)\tilde{f}(\tau)d\tau$$ 
fulfils
\begin{equation}
\label{est-u-left-n>2}
|u(t)|+|u'(t)|\le c e^{\frac{n+2}{2}t}\|g\|_{C^{0,\beta}_3(\Sigma)}\qquad\forall t\in(-\infty, t_0),\, n\ge 3
\end{equation}
and
\begin{equation}
\label{est-u-left-n=2}
|u(t)|+|u'(t)|\le cte^{2t}\|g\|_{C^{0,\beta}_3(\Sigma)}\qquad\forall t\in(-\infty, t_0),\, n=2
\end{equation}

\subsection{The intermediate interval}\label{subs-intermediate}
We fix $t_1>t_0$ and we solve equation (\ref{Jacobi-eq-EF}) in the interval $(t_0,t_1)$. The value of $t_1$ will be fixed in Subsection \ref{subs-right}. We consider a fundamental set $\{v_\pm\}$ such that $$v_+(t_0)=v'_-(t_0)=1,\qquad v_-(t_0)=v'_+(t_0)=0$$
and we define the solution $v$ through the variation of parameter formula
$$v(t)=u(t_0)v_+(t)+u'(t_0)v_-(t)+v_+(t)\int_{t_0}^t v_-(\tau)\tilde{f}(\tau)d\tau-v_-(t)\int_{t_0}^t v_+(\tau)\tilde{f}(\tau)d\tau,$$
so that $v(t)$ solves equation (\ref{Jacobi-eq-EF}) in $(t_0,t_1)$ with initial values $v(t_0)=u(t_0)$ and $v'(t_0)=u'(t_0)$. Thus, $v\in C^{2,\beta}$ extends of $u$ to $(-\infty,t_1)$. Since $v_\pm$ and $\tilde{f}$ are bounded in $(t_0,t_1)$, then 
\begin{equation}
\label{est-u-intermediate}
|v(t)|+|v'(t)|\le c\|g\|_{C^{0,\beta}_3(\Sigma)}\qquad\forall\, t\in(t_0,t_1).
\end{equation}

\subsection{The interval $(t_1,\infty)$}\label{subs-right}
Due to the asymptotic behaviour of the potential at infinity, for $t_1>0$ sufficiently large it is possible to find a fundamental set $w_\pm$ such that $$w_+(t_1)=w'_-(t_1)=1,\qquad w_-(t_1)=w'_+(t_1)=0$$
and 
$$w_\pm,\,\partial_t w_\pm=O(1)$$
as $t\to\infty$. Applying once again the variation of parameters formula we can define a solution to (\ref{Jacobi-eq-EF}) by setting
$$w(t)=v(t_0)w_+(t)+v'(t_0)w_-(t)+w_+(t)\int_{t_0}^t w_-(\tau)\tilde{f}(\tau)d\tau-w_-(t)\int_{t_0}^t w_+(\tau)\tilde{f}(\tau)d\tau.$$
As above, the initial conditions guarantee that $w$ is an extension of $u$ to an entire $C^{2,\beta}$ solution to (\ref{Jacobi-eq-EF}). Moreover, the asymptotic behaviour of $f$ at infinity gives the estimate for $t\in(t_1,\infty)$,
\begin{equation}
\label{est-u-intermediate}
|w(t)|+|w'(t)|\le 
\begin{cases}
c e^{\frac{N-4}{2}t}\|g\|_{C^{0,\beta}_3(\Sigma)}\qquad& N=5,6\\
ct\|g\|_{C^{0,\beta}_3(\Sigma)}\qquad& N=4\\
c\|g\|_{C^{0,\beta}_3(\Sigma)}\qquad& N=3
\end{cases}
\end{equation}
\subsection{The behaviour of the solution at infinity}
To summarise, we pullback to the $s$-variable, to find that the asymptotic behaviour of $\psi(s)=\psi(s)=p(\ln(s))u(\ln(s))$ is given by
\begin{equation}\notag
\begin{aligned}
&|\psi(s)|(s+1)\le c\|g\|_{C^{0,\beta}_3(\Sigma)},\,|\psi_s(s)|(s+1)^2\le c\|g\|_{C^{0,\beta}_3(\Sigma)}\qquad\forall\, s\ge 0,\,N=5,6\\
&|\psi(s)|\frac{s+1}{\log(s+2)}\le c\|g\|_{C^{0,\beta}_3(\Sigma)},\,|\psi_s(s)|\frac{(s+1)^2}{\log(s+2)}\le c\|g\|_{C^{0,\beta}_3(\Sigma)}\qquad\forall\, s\ge 0,\,N=4\\
&|\psi(s)|\frac{(s+1)^{1/2}}{\log(s+2)}\le c\|g\|_{C^{0,\beta}_3(\Sigma)},\,|\psi_s(s)|\frac{(s+1)^{3/2}}{\log(s+2)}\le c\|g\|_{C^{0,\beta}_3(\Sigma)}\qquad\forall\, s\ge 0,\,N=3.
\end{aligned}
\end{equation}
We note that, in a neighbourhood of the origin, the asymptotic behaviour of the solution is given by
\begin{equation}
\begin{aligned}
\psi(s)=O(s^2)
\qquad\text{as $s\to 0^+$, $n\ge 3$}\\
\psi(s)=O(s^2 |\log s|)
\qquad\text{as $s\to 0^+$, $n=2$}
\end{aligned}
\end{equation}
As a consequence, the corresponding solution $\phi(y):=\psi(s)$ to \eqref{eq-trA3} fulfills \eqref{est-phi-sharp}.\\



\begin{remark}
In conclusion, for $N=3,4$, the ODE analysis gives a better estimate of the asymptotic behaviour of $\phi$ than the one predicted by applying point $(1)$ of Theorem \ref{prop-right-inverse-Jacobi}. Roughly speaking, we can see that case $N=3,4$, we can see that the presence of an indicial root makes us lose a logarithm.
\end{remark}

\bibliographystyle{abbrv}
	\bibliography{bibliography}
	\addcontentsline{toc}{chapter}{\protect\numberline{}Bibliography}
\end{document}